%% file: main.tex
\documentclass[12pt]{amsart}
\usepackage{preamble}
\begin{document}
	
\title{From category $\cO^\infty$ to locally analytic representations}
\author{Shishir Agrawal}
\address{Colorado College, Department of Mathematics $\&$ Computer Science, 14 East Cache la Poudre St., Colorado Springs, CO, 80903 }
\email{sagrawal@coloradocollege.edu}
\author{Matthias Strauch}
\address{Indiana University, Department of Mathematics, Rawles Hall, Bloomington, IN 47405, U.S.A.}
\email{mstrauch@indiana.edu}
	

\begin{abstract}
Let $G$ be a $p$-adic reductive group and $\mathfrak{g}$ its Lie algebra. We construct a functor from the extension closure of the Bernstein-Gelfand-Gelfand category $\mathcal{O}$ associated to $\mathfrak{g}$ into the category of locally analytic representations of $G$, thereby expanding on an earlier construction of Orlik-Strauch. A key role in this new construction is played by $p$-adic logarithms on tori. This functor is shown to be exact with image in the subcategory of admissible representations in the sense of Schneider and Teitelbaum. En route, we establish some basic results in the theory of modules over distribution algebras and related subalgebras, such as a tensor-hom adjunction formula. We also relate our constructions to certain representations constructed by Breuil and Schraen in the context of the $p$-adic Langlands program. 
\end{abstract}

\maketitle
\tableofcontents

\input{introduction}

\part*{Part I}

\input{preliminaries}

\input{lifting}

\input{categories-o-infinity}

\input{functor}

\input{examples}

\part*{Part II}

\input{adjunction}

\input{modules}

\bibliographystyle{alphaurl}
\bibliography{mybib}

\end{document}

%% file: introduction.tex
\section*{Introduction}

Let $F$ be a finite extension of $\Qp$, and let $\bbG \supseteq \bbP \supseteq \bbB \supseteq \bbT$ be a connected split reductive group over $F$, a parabolic subgroup, a Borel subgroup, and a maximal torus, respectively. Denote by $\frg, \frp$, $\frb$, and $\frt$ their respective Lie algebras, and let $G, P, B$, and $T$ be their groups of $F$-points, regarded as locally $F$-analytic groups. Inside the Bernstein-Gelfand-Gelfand category $\cO$ \cite{BGG2}, one has the category $\cO^{\frp}_\alg$ the category of finitely generated $\frg$-modules which are locally finite dimensional over $\frp$ and on which $\frt$ acts semisimply with algebraic weights. The paper \cite{OrlikStrauchJH} introduced exact functors $\cF^G_P$ from $\cO^{\frp}_\alg$ to the category of admissible locally analytic representations of $G$. 

As is well-known, the category $\cO$ is not stable under extensions in the category of all $\frg$-modules. Its extension closure is often denoted $\cO^\infty$ and is sometimes called the ``thick category $\cO$'' \cite{SoergelDiploma,CoulembierMazorchuckIII}. Explicitly, objects in this category are still finitely generated over $\frg$ and the action of $\frb$ is still locally finite, but the action of $\frt$ is not necessarily semisimple (cf. \cref{cat-o-infinity}). Correspondingly, the subcategory $\cO^{\frp}_\alg$ is also not stable under extensions, and we denote by $\cO^{\frp,\infty}_\alg$ its extension closure. The object of this paper is to extend the functors $\cF^G_P$ from $\cO^{\frp}_\alg$ to its extension closure $\cO^{\frp,\infty}_\alg$. In fact, we construct a family of extensions, indexed by ``logarithms'' on $T$ (i.e., homomorphisms $T \to \frt$ which invert the exponential $\exp : \frt \dashedarrow T$, cf. \cref{log-definition}). 

In \cite{OrlikStrauchJH}, the strategy behind defining the functor $\cF^G_P$ on category $\cO^{\frp}_\alg$ is as follows. We first set 
\[\vcF^G_P(M) = D(G) \otimes_{D(\frg,P)} M, \]
where $D(G)$ is the locally analytic distribution algebra \cite{ST02b,ST03} and $D(\frg,P)$ is the subring of $D(G)$ generated by $U(\frg)$ and the distribution algebra $D(P)$ of $P$. The key point here is that the action of $\frp$ on $M$ lifts functorially to an action of $P$, and thus gives rise to a structure of a $D(P)$-module on $M$ which is compatible with the structure of a $\frg$-module. It turns out that $\vcF^G_P(M)$ is a coadmissible $D(G)$-module, and we define $\cF^G_P(M)$ to be its continuous dual. 

To extend this functor to $\cO^{\frp,\infty}_\alg$, we must first construct a functorial lift of the action of $\frp$ to an action of $P$, which in turn depends on the choice of a logarithm $\log : T \to \frt$ . Having made this choice, we can lift a nilpotent action $\phi : \frt \to \End(V)$ on a finite dimensional vector space $V$ to an action of $T$ that is given by the composite $\exp \circ\, \phi \circ \log$ (cf. \cref{lifting-nilpotent}). By twisting this action appropriately for other generalized weight spaces, we can lift arbitrary algebraic actions of $\frt$ on finite dimensional vector spaces to actions of $T$. It turns out that this lifted action is suitably natural and yields an action of $P$ (cf. \cref{lifting-general}). We thus obtain a $D(\frg,P)$-module structure on any $M \in \cO^{\frp,\infty}_\alg$ which depends on the choice of logarithm $\log : T \to \frt$. We denote $M$ equipped with this $D(\frg,P)$-module structure by $\Lift(M, \log)$. 

\begin{utheorem}
Suppose $M \in \cO^{\frp,\infty}_\alg$ and $V$ is a smooth strongly admissible representation of $P$.\footnote{In other words, $V$, when regarded as a $P_0$-representation for some compact open subgroup $P_0 \subseteq P$, embeds into $C^\infty(P_0)^{\oplus m}$ for some $m \geq 0$ (cf. \cite[p. 115]{ST01a} or \cite[p. 453]{ST02b}).}
Define
\[ \vcF^G_P(M, V) = \vcF^G_{P,\log}(M, V) = D(G) \otimes_{D(\frg,P)} \Big(\Lift(M, \log) \otimes V'\Big) \]
and set $\cF^G_P(M,V) = \vF^G_P(M,V)'$. 
\begin{enumerate}[(i)]
\item (\Cref{coadmissibility}) $\cF^G_P(M,V)$ is an admissible representation of $G$. 
\item (\Cref{exactness}) The bifunctor $(M, V) \mapsto \cF^G_P(M,V)$ is exact in both arguments. 
\item (\Cref{PQthm}) If $\bbQ$ is a parabolic subgroup containing $\bbP$, $M \in \cO^{\frq,\infty}_\alg$ and $V$ is a strongly admissible smooth representation of $P$, then there is an isomorphism of $D(G)$-modules 
\[ \vcF^G_P(M,V) = \vcF^G_Q(M,\ind^Q_P(V)). \]
where $\ind^P_Q(V)$ is the smooth induction of $V$ from $P$ to $Q$. 
\end{enumerate}
\end{utheorem}

While our construction generalizes the work done in \cite{OrlikStrauchJH}, we here build the theory from the ground up and do not rely on arguments given in \cite{OrlikStrauchJH}. Indeed, the proofs of the main results are based on new techniques in Schneider-Teitelbaum's theory of $D(G)$-modules which we develop especially in the part II of this paper. 

Locally analytic representations associated to objects in the category $\cO^{\frp,\infty}_\alg$ via our functor appear naturally in the $p$-adic Langlands program (cf. \cref{examples}). In \cite{BreuilInv} and \cite{SchraenGL3}, the authors construct locally analytic representations of $\GL_2(\Qp)$ and $\GL_3(\Qp)$, respectively, which are related to semistable Galois representations of dimension 2 and 3, respectively, which are not potentially crystalline. Indeed, if $V$ is such a 2-dimensional Galois representation, the representation $\Sigma(k,\sL)$ of \cite{BreuilInv} appears as a subrepresentation of the space of locally analytic vectors $\Pi(V)^\an$ of Colmez's $p$-adic Banach space representation $\Pi(V)$. This $\Sigma(k,\sL)$ is in turn a subquotient of a representation which is in the image of our functor $\cF^{\GL_2(\Qp)}_{B,\log}$, where the logarithm $\log$ used depends on the $\sL$-invariant of $V$. Similarly, the representation $\Sigma(\lambda,\sL,\sL')$ of \cite{SchraenGL3} is a subquotient of a representation in the image of our functor $\cF^{\GL_3(\Qp)}_{B,\log}$, where, again, the logarithm $\log$ used depends on the pair of $\sL$-inariants $(\sL,\sL')$ of a certain semistable Galois representation. 

Extending the functors $\cF^G_P$ from category $\cO^\frp_\alg$ to its extension closure $\cO^{\frp,\infty}_\alg$ is necessary if one wants to work with these categories and functors in the setting of derived categories. On the side of locally analytic representations, the paper \cite{ST05} considers the bounded derived category $D^b_{\cC_G}(D(G))$ of complexes of $D(G)$-modules whose cohomology objects are coadmissible. The natural pendant to this category on the side of Lie algebra representations is then the bounded derived category $D^b_{\cO^{\frp,\infty}_\alg}(U(\frg))$ of complexes of $U(\frg)$-modules whose cohomology objects are in category $\cO^{\frp,\infty}_\alg$. Here, in order for this category to be a triangulated category, it is indeed necessary to work with $\cO^{\frp,\infty}_\alg$, and not only with $\cO^\frp_\alg$, as the latter category is not closed under extensions. Moreover, by \cite{CoulembierMazorchuckIII}, the natural functor $D^b(\cO^{\frp,\infty}_\alg) \ra D^b_{\cO^{\frp,\infty}_\alg}(U(\frg))$ is an equivalence of categories, which makes it possible to define the functors $\vcF^G_P$ on $D^b_{\cO^{\frp,\infty}_\alg}(U(\frg))$. The study of the functors $\vcF^G_P$ in the setting of derived categories will be taken up in forthcoming work.

\subsection*{Structure of the paper}

The main constructions and results described above are developed in part I of this paper. The proofs draw occasionally on work carried out in part II. On the other hand, part II does not rely on part I. It contains several results about modules over distribution algebras and certain subalgebras which may be of independent interest. 

\subsection*{Acknowledgments}

We would like to thank Simon Wadsley and the anonymous referee for carefully reading earlier versions of this paper and for their helpful questions and remarks. 

\subsection*{Notation}

Modules are assumed to be left modules unless otherwise specified. 

Let $F$ denote a finite extension of $\Q_p$, and $E$ a finite extension of $F$. If $X$ is an object over $F$ (vector space, algebra, etc), we denote by $X_E$ its base change to $E$. Vector spaces, tensor products, homs, etc, are assumed to be over $E$ unless otherwise specified. If $X$ is an object of a category with a faithful functor into vector spaces, we say that $X$ is \emph{locally finite dimensional} if it is a colimit of its finite dimensional subobjects. 

If $\frg$ is a Lie algebra over $F$, we write $U(\frg)$ in place of $U(\frg_E)$. Also, a \emph{$\frg$-module} is an $E$-vector space $M$ equipped with a Lie algebra homomorphism $\phi_M : \frg \to \End_E(M)$ (in other words, we use the term \emph{$\frg$-module} to refer to what might more accurately be called a $\frg_E$-module). For $m \in M$ and $x \in \frg$, we sometimes write $x \bdot m$ in place of $\phi_M(x)(m)$. Also, when $M$ is clear from context, we sometimes drop the subscript from $\phi_M$ and simply write $\phi$ instead. We let $\Mod_\frg$ denote the $E$-linear category of $\frg$-modules. We write $\Mod_\frg^*$ for various $*$ to denote various full subcategories of $\Mod_\frg$, as introduced below. 

We use upper case letters ($G, P$, etc) to denote locally analytic groups over $F$, and we use corresponding lower case fraktur letters ($\frg, \frp$, etc) to denote their Lie algebras. If $G$ is a locally analytic group, we write $\Rep_G^\an$ for the category of locally analytic representations of $G$ with coefficients in $E$. We write $C^\an(G), C^\infty(G)$, and $D(G)$ for the algebras of analytic functions, smooth functions, and analytic distributions on $G$ with coefficients in $E$. 

We denote blackboard bold letters ($\bbG, \bbP$, etc) to denote algebraic groups over $F$. In this case, we then use the corresponding upper case letter ($G, P$, etc) to denote the group of $F$-points of the algebraic group regarded as a locally analytic group, and the corresponding lower case fraktur letter ($\frg, \frp$, etc) to denote its Lie algebra. 

%% file: preliminaries.tex
\section{Preliminaries}

For the sake of completness, in this section we recall some well-known definitions and results for which we did not have an adequate reference. 

\subsection{Locally finite dimensional representations}

Let $\frg$ be a Lie algebra over $F$. We write $\Mod_{\frg}^{\fd}$  for the category of finite dimensional $\frg$-modules, and $\Mod_{\frg}^{\lfd}$ for the category of locally finite dimensional $\frg$-modules (ie, the ones for which every finitely generated $\frg$-submodule is finite dimensional). 

\begin{lemma} \label{fd-lfd-serre}
$\Mod_{\frg}^{\fd}$ and $\Mod_{\frg}^{\lfd}$ are both Serre subcategories of $\Mod_\frg$. 
\end{lemma}

\begin{proof}
Suppose 
\[ \begin{tikzcd} 0 \ar{r} & M' \ar{r} & M \ar{r} & M'' \ar{r} & 0 \end{tikzcd} \]
is an exact sequence of $\frg$-modules. It is clear that $M$ is finite dimensional if and only if $M'$ and $M''$ are. It is also clear that if $M$ is locally finite dimensional, then $M'$ and $M''$ are both locally finite dimensional as well. Conversely, suppose $M'$ and $M''$ are locally finite dimensional. Let $N$ be a finitely generated submodule of $M$. Then we have an exact sequence as follows. 
\[ \begin{tikzcd} 0 \ar{r} & M' \cap N \ar{r} & N \ar{r} & N/(M' \cap N) \ar{r} & 0 \end{tikzcd} \] 
Since $N/(M' \cap N)$ is a finitely generated submodule of $M''$, it is finite dimensional. Since $U(\frg)$ is noetherian and $N$ is finitely generated, $M' \cap N$ is also finitely generated; so, since $M'$ is locally finite dimensional, $M' \cap N$ must be finite dimensional. Thus, since $M' \cap N$ and $N/(M' \cap N)$ are both finite dimensional, $N$ must be as well. 
\end{proof}

\subsection{Modules over an abelian Lie algebra}

Let $\frt$ be an abelian Lie algebra over $F$. Let $E[X]$ be the polynomial ring in an indeterminate $X$ over $E$. We identify points of $\MaxSpec(E[X])$ with their monic generators. Note that, if $\bar{E}$ is an algebraic closure of $E$, then there is a surjective function $\bar{E} \to \MaxSpec(E[X])$ which maps elements of $\bar{E}$ to their minimal polynomials.

\begin{definition}
Let $\frt_E^\sharp$ denote the set of functions $\pi : \frt_E \to \MaxSpec(E[X])$ which factor through an $E$-linear map $\frt_E \to \bar{E}$ for any algebraic closure $\bar{E}$ of $E$. Note that if this is true for one choice of algebraic closure, it is true for any other choice as well. 
\end{definition}

\begin{definition}
Let $M$ be a $\frt$-module. If $\pi \in \frt_E^\sharp$, the \emph{primary component} $M_\pi$ of $M$ associated to $\pi$ is 
\[ M_\pi = \left\{ m \in M : (\pi(x)(\phi(x)))^n(m) = 0 \text{ for some } n \geq 0 \right\}, \]
where $x \in \frt$ and $\pi(x)(\phi(x))$ is the endomorphism of $M$ obtained by plugging $\phi(x) \in \End(M)$ into the monic irreducible polynomial $\pi(x) \in \MaxSpec(E[X])$. 
\end{definition}

\begin{theorem} \label{primary-decomposition}
For any locally finite dimensional $\frt$-module $M$, we have 
\[ M = \bigoplus_{\pi \in \frt_E^\sharp} M_\pi. \]
This is called the \emph{primary component decomposition} of $M$. 
\end{theorem}

\begin{proof}
When $M$ is finite dimensional, this is the statement of \cite[chapter II, section 4, theorem 5; p. 40--41]{JacobsonLie}. In general, observe that we have \[ N = \bigoplus_{\pi \in \frt_E^\sharp} N_\pi \] for any finite dimensional submodule of $M$. Since $M$ is the colimit over all such finite dimensional submodules $N$, and since colimits commute with direct sums, we have 
\[ M = \colim N = \colim \left( \bigoplus_{\pi \in \frt_E^\sharp} N_\pi \right) = \bigoplus_{\pi \in \frt_E^\sharp} \colim N_\pi = \bigoplus_{\pi \in \frt_E^\sharp} M_\pi. \qedhere \]
\end{proof}

\begin{proposition} \label{primary-component-exact}
$M \mapsto M_\pi$ defines an exact functor on $\Mod_{\frt}^\lfd$ for any $\pi \in \frt_E^\sharp$. 
\end{proposition}

\begin{proof}
One can check directly from the definition of the primary component that, if $\sigma : M \to N$ is a homomorphism of $\frt$-modules, then $\sigma(M_\pi) \subseteq N_\pi$, which shows that $M \mapsto M_\pi$ is in fact functorial. For exactness, suppose we have an exact sequence of $\frt$-modules
\[ \begin{tikzcd} M' \ar{r}{\sigma} & M \ar{r}{\tau} &  M'' \end{tikzcd} \]
and suppose $m \in \ker(\tau) \cap M_\pi$. Since $m \in \ker(\tau)$, there exists $m' \in M'$ such that $\sigma(m') = m$. Since $M'$ has a primary component decomposition by \cref{primary-decomposition}, we can write $m' = m'_1 + \dotsb + m'_n$ where $m'_i \in M'_{\pi_i}$ for distinct $\pi_i \in \frt_E^\sharp$. Then 
\[ m = \sigma(m') = \sigma(m'_1) + \dotsb + \sigma(m'_n) \]
and we have $\sigma(m'_i) \in M_{\pi_i}$. Since $m \in M_\pi$, there must exist a single $k$ such that $m = \sigma(m_k)$ and $\pi = \pi_k$ and $\sigma(m'_i) = 0$ for all $i \neq k$. Thus $m \in \sigma(M'_\pi)$, proving that
\[ \begin{tikzcd} M'_\pi \ar{r}{\sigma} & M_\pi \ar{r}{\tau} &  M''_\pi \end{tikzcd} \]
is exact.   
\end{proof}

\begin{definition}
There is a natural embedding $\frt_E^* = \Hom(\frt_E, E) \hookrightarrow \frt_E^\sharp$ which carries $\lambda \in \frt_E^*$ to the function $\pi_\lambda(x) = X - \lambda(x) \in \MaxSpec(E[X])$. For a locally finite dimensional $\frt$-module $M$, we write $M_\lambda$ instead of $M_{\pi_\lambda}$ and call $M_\lambda$ the \emph{generalized weight space} of $M$ associated to $\lambda$. Explicitly, we have
\[ M_\lambda = \{ m \in M : (\phi(x) - \lambda(x))^n(m) = 0 \text{ for some } n \geq 0 \}. \]
\end{definition}

\begin{definition} \label{definition-split}
A $\frt$-module $M$ is \emph{split} if it is locally finite dimensional and $M_\pi = 0$ unless $\pi = \pi_\lambda$ for some $\lambda \in \frt_E^*$. In other words, $M$ is split if it is locally finite dimensional and there exists a $\frt$-module decomposition
\begin{equation} \label{gws-decomposition} M = \bigoplus_{\lambda \in \frt_E^*} M_\lambda. \end{equation}
This decomposition is called the \emph{generalized weight space decomposition} of $M$. We write $\Mod_{\frt}^s$ for the category of split $\frt$-modules. 
\end{definition}

Note that every finite dimensional $\frt$-module becomes split after a finite field extension \cite[chapter II, section 4]{JacobsonLie}. 

\begin{lemma} \label{split-serre}
$\Mod_\frt^s$ is a Serre subcategory of $\Mod_\frt$. 
\end{lemma}

\begin{proof}
Suppose 
\[ \begin{tikzcd} 0 \ar{r} & M' \ar{r} & M \ar{r} & M'' \ar{r} & 0 \end{tikzcd} \]
is an exact sequence of $\frt$-modules. We know from \cref{fd-lfd-serre} that $M$ is locally finite dimensional if and only if $M'$ and $M''$ are. Now if $\pi \in \frt_E^\sharp$ and $\pi \neq \pi_\lambda$ for any $\lambda \in \frt_E^*$, then we know that 
\[ \begin{tikzcd} 0 \ar{r} & M'_\pi \ar{r} & M_\pi \ar{r} & M''_\pi \ar{r} & 0 \end{tikzcd} \]
is exact by \cref{primary-component-exact}, from which we see that $M_\pi = 0$ if and only if $M'_\pi = 0$ and $M''_\pi = 0$. 
\end{proof}

\begin{definition}
Let $M$ be a split $\frt$-module. Let $\phi_s : \frt \to \End(M)$ be the map given by $\phi_s(x) : m \mapsto \lambda(x)m$ on $M_\lambda$ for each $\lambda$. Then $\phi_s$ is a Lie algebra homomorphism, and it is called the \emph{semisimple part} of $\phi$. The \emph{nilpotent part} of $\phi$ is $\phi_n = \phi - \phi_s$.
\end{definition}

For the remainder of this subsection, we regard $\frt$ as the Lie algebra of an algebraic torus $\bbT$. 

\begin{definition}
We say that $\lambda \in \frt^*_E$ is \emph{algebraic} if there exists an algebraic character $\chi_\lambda : \bbT_E \to \bbG_{m,E}$ which differentiates to $\lambda$, in the sense that 
\[ \lambda(x) = \left. \frac{d}{dt} \chi_\lambda(\exp(tx)) \right|_{t = 0}. \]
Such a character $\chi$, if it exists, is uniquely determined by $\lambda$. 
\end{definition} 

\begin{definition}
A split $\frt$-module $M$ \emph{has algebraic weights} if $M_\lambda \neq 0$ for some $\lambda \in \frt_E^*$  only if $\lambda$ is algebraic. We write $\Mod_\frt^{\alg}$ for the category of split $\frt$-modules which have algebraic weights, and $\Mod_{\frt}^{\alg,\fd}$ for the subcategory of finite dimensional ones. 
\end{definition}

\begin{lemma} \label{alg-serre}
$\Mod_{\frt}^\alg$ and $\Mod_{\frt}^{\alg,\fd}$ are both Serre subcategories of $\Mod_{\frt}$. 
\end{lemma}

\begin{proof}
This follows from \cref{fd-lfd-serre,primary-component-exact}, similar to the proof of \cref{split-serre}. 
\end{proof}

\subsection{Nilpotent actions of Lie algebras}

Let $\fru$ be a Lie algebra. 

\begin{definition} \label{nilp-definition}
Let $M$ be a $\fru$-module. We say that \emph{$\fru$ acts nilpotently on $M$} if every $x \in \fru$ acts on $M$ by a nilpotent endomorphism; in other words, if for every $x \in \fru$, there exists a positive integer $n$ such that $x^n \bdot m = 0$ for all $m \in M$. We say that \emph{$\fru$ acts locally nilpotently on $M$} if for every $x \in \fru$ and $m \in M$, there exists a positive integer $n$ such that $x^n \bdot m = 0$. Let $\Mod_\fru^\unilp$ denote the category of $\fru$-modules on which $\fru$ acts locally nilpotently.  
\end{definition}

\begin{lemma} \label{nilp-serre}
$\Mod_\fru^\unilp$ is a Serre subcategory of $\Mod_\fru$. 
\end{lemma}

\begin{proof}
It is clear that, if $\fru$ acts locally nilpotently on $M$, then the same is true of any subquotient of $M$. Suppose 
\[ \begin{tikzcd} 0 \ar{r} & M' \ar{r} & M \ar{r} & M'' \ar{r} & 0 \end{tikzcd} \]
is an exact sequence of $\fru$-modules. Fix $x \in \fru$ and $m \in M$. There exists some integer $n_1$ which annihilates the image of $m$ in $M''$, which means that $x^{n_1} \bdot m \in M'$. Then there exists some other integer $n_2$ which annihilates $x^{n_1} \bdot m$. This means that $x^{n_2} \bdot (x^{n_1} \bdot m) = x^{n_2 + n_1} \bdot m = 0$. 
\end{proof}

\subsection{Rational points of a connected reductive group}

Let $(\bbG, \bbT)$ be a connected split reductive algebraic group over $F$. Let $\Phi(\bbG, \bbT)$ denote the corresponding root system. Recall our standing convention that when an algebraic group is denoted in blackboard bold, we use a corresponding upper case letter to denote its group of $F$-points. Note that the following statement is a statement about groups of $F$-points (rather than about algebraic groups, or equivalently, about the group of $\bar{F}$-points). 

\begin{lemma} \label{generated-by-torus-and-root-subgroups}
The group $G$ is generated by $T$ and $U_\alpha$ for all $\alpha \in \Phi(\bbG, \bbT)$, where $\bbU_\alpha$ is the root group of $\bbG$ corresponding to $\alpha$. 
\end{lemma}

\begin{proof}
Fix a Borel subgroup $\bbB$ containing $\bbT$ and let $\bbU$ be its unipotent radical. The multiplication map $\prod_{\alpha > 0} \bbU_\alpha \to \bbU$ is an isomorphism of $F$-schemes \cite[2.3]{BorelTits}, so on $F$-points, we see that $U$ is generated by $U_\alpha$ for all $\alpha > 0$. Let $\bbN$ be the normalizer of $\bbT$ in $\bbG$. Then $G = UNU$ \cite[2.11]{BorelTits}, so it suffices to show that every element $n \in N$ is in the subgroup of $G$ generated by $T$ and $U_\alpha$ for all $\alpha$. 

The quotient $N/T$ is the Weyl group of $\bbG$, so it is generated by the reflections across hyperplanes perpendicular to $\alpha$ for all $\alpha$. Thus it suffices to fix $\alpha \in \Phi(\bbG, \bbT)$ and assume that $n$ induces reflection across the hyperplane perpendicular to $\alpha$. Fix a nontrivial $g \in U_\alpha$. Then 
\[ N \cap U_{-\alpha} g U_{-\alpha} = \{ m(g) \} \]
for some $m(g)$ which also induces reflection across the hyperplane perpendicular to $\alpha$ \cite[lemma 0.19]{Landvogt}. Since both $m(g)$ and $n$ induce the same element of the Weyl group, there exists $t \in T$ such that $n = m(g)t$. Since $m(g) \in U_{-\alpha} g U_{\alpha}$, which in turn is contained in the subgroup generated by $U_\alpha$ and $U_{-\alpha}$, this concludes the proof. 
\end{proof}

\begin{corollary} \label{generated-by-torus-and-derived}
The group $G$ is generated by $T$ and $G'$, where $\bbG'$ is the derived subgroup of $\bbG$. 
\end{corollary}

\begin{proof}
After \cref{generated-by-torus-and-root-subgroups}, it is sufficient to show that each root group  $\bbU_\alpha$ of $\bbG$ is contained in $\bbG'$. Since $\frg' = \Lie(\bbG')$ is the derived subalgebra of $\frg$, it contains $\frg_\alpha$. Let $\bbU_\alpha'$ be the root group of $\mathbb{G}'$ corresponding to $\frg_\alpha$. Then $\bbU_\alpha'$ and $\bbU_\alpha$ are both closed connected subgroups of $\bbG$ with Lie algebra $\frg_\alpha$, so $\bbU_\alpha = \bbU_\alpha'$ \cite[theorem 13.1]{Humphreys_Groups}. Thus $\bbU_\alpha$ is in $\bbU'$.
\end{proof}

%% file: lifting.tex
\section{Lifting Lie algebra actions} \label{lifting}

In this section, we describe ways of lifting actions of Lie algebras over $F$ to actions of locally analytic groups. 

\subsection{Lifting to actions of unipotent groups} \label{lifting-unipotent}

The easiest situation is lifting actions of nilpotent Lie algebras to actions of corresponding unipotent locally analytic groups. 

\begin{para}
Let $\bbU$ be a unipotent algebraic group. Then $\fru$ is nilpotent, the exponential map $\exp : \fru \to U$ is bijective, and we write $\log : U \to \fru$ to denote its inverse. Suppose $M$ is a finite dimensional $\fru$-module on which $\fru$ acts nilpotently, ie, such that $\phi(\fru) \subseteq \Nil(M)$, where $\Nil(M)$ denotes the set of nilpotent endomorphisms of $M$ (cf. \cref{nilp-definition}). We let $\tilde{\phi}_M$ denote the composite
\[ \begin{tikzcd} U \ar{r}{\log} & \fru \ar{r}{\phi} & \Nil(M) \ar{r}{\exp} & \GL(M). \end{tikzcd} \]
Then $\tilde{\phi}_M$ defines a locally analytic action of $U$ on $M$ (in fact, the action is even algebraic). We write $\Lift(M)$ to denote $M$ regarded as a locally analytic representation of $U$. When $M$ is clear from context, we drop the subscript and simply write $\tilde{\phi}$ instead of $\tilde{\phi}_M$. \Cref{lifting-unipotent-check} below shows that $\tilde{\phi}$ does in fact lift the original action of $\fru$ on $M$.
\end{para}

\begin{lemma} \label{lifting-unipotent-check}
Suppose $M$ is a finite dimensional $\fru$-module on which $\fru$ acts nilpotently. Then
\[ \phi_M(x) = \left. \frac{d}{dt} \tilde{\phi}_M(\exp(tx)) \right|_{t = 0} \]
for all $x \in \fru$. 
\end{lemma}

\begin{proof}
Observe that
\[ \tilde{\phi}(\exp(tx)) = \exp(\phi(\log(\exp(tx)))) 
= \exp(\phi(tx)) 
= \exp(t\phi(x))  \]
which means that 
\[ \left. \frac{d}{dt} \tilde{\phi}(\exp(tx)) \right|_{t = 0} = \left. \frac{d}{dt} \exp(t\phi(x)) \right|_{t = 0} = \exp(t\phi(x))\phi(x) \bigg|_{t = 0} = \phi(x). \qedhere \]
\end{proof}

\begin{para}
This construction is evidently natural and defines a functor \[ \Lift : \Mod_\fru^{\unilp,\fd} \to \Rep_U^\an \] on the category $\Mod_\fru^{\unilp,\fd}$ of finite dimensional $\fru$-modules on which $\fru$ acts nilpotently.
\end{para}

\subsection{Logarithms on tori}

Let $\bbT$ be a split algebraic torus and let $T_0$ be the maximal compact subgroup of $T = \bbT(F)$. 

\begin{definition} \label{log-definition}
A \emph{logarithm} on $T$ is a locally analytic group homomorphism $T \to \frt_E$ such that, on a neighborhood of 0 where $\exp : \frt \dashedarrow T$ is defined, the composite $\log \circ \exp$ equals the natural map $\frt \to \frt_E$ (we will sometimes abusively write $\log \circ \exp = \id$ for brevity). We write $\Logs(T)$ for the set of all logarithms on $T$. 
\end{definition}

\begin{lemma}
$\Logs(T)$ is a torsor for $\Hom(T/T_0, \frt_E)$. 
\end{lemma}

\begin{proof}
Let $T_f$ be the set of $x \in T$ such that $x^n \to 1$ for some strictly increasing sequence of integers $n$, as in \cite[III.6, proposition 10]{B-L}. Then $T_f$ is the union of all compact subgroups of $T$ \cite[III.6, corollary to proposition 13]{B-L}, so $T_f = T_0$ since $T_0$ is the unique maximal compact subgroup of $T$. By \cite[III.6, propositions 10--11]{B-L}, there exists a unique group homomorphism $\lambda : T_0 \to \frt_E$ which inverts the exponential map on a neighborhood of 0 in $\frt_E$. In other words, $\Logs(T)$ is precisely the set of locally analytic homomorphisms $T \to \frt_E$ which restrict to $\lambda$ on $T_0$. 

Since $T/T_0$ is a discrete, finite free abelian group, the exact sequence
\[ \begin{tikzcd} 1 \ar{r} & T_0 \ar{r} & T \ar{r} & T/T_0 \ar{r} & 1 \end{tikzcd} \]
splits in the category of abelian Lie groups, so applying $\Hom(-,\frt_E)$ yields an exact sequence 
\begin{equation} \label{torsor-exact-sequence} \begin{tikzcd} 0 \ar{r} & \Hom(T/T_0, \frt_E) \ar{r} & \Hom(T,\frt_E) \ar{r} & \Hom(T_0,\frt_E) \ar{r} & 0. \end{tikzcd} \end{equation}
Then $\Logs(T) \subseteq \Hom(T,\frt_E)$ is precisely the preimage of $\lambda \in \Hom(T_0, \frt_E)$, so it is a torsor over the kernel $\Hom(T/T_0, \frt_E)$.
\end{proof}

\begin{example}
Suppose $\bbT = \bbG_m$, so that $T = F^\times$ and $T_0 = \fro_F^\times$. The exact sequence
\[ \begin{tikzcd} 1 \ar{r} & \fro_F^\times \ar{r} & F^\times \ar{r} & F^\times/\fro_F^\times \ar{r} & 1 \end{tikzcd} \]
splits: if we choose a uniformizer $\varpi \in F$, then the subgroup $\varpi^{\bbZ} \subseteq F^\times$ maps isomorphically onto $F^\times/\fro_F^\times$. Now note that we have another exact sequence 
\[ \begin{tikzcd} 1 \ar{r} & T_0^\circ \ar{r} & \fro_F^\times \ar{r} & \fro_F^\times/T_0^\circ \ar{r} & 1 \end{tikzcd} \]
where $T_0^\circ = \{x \in F \mid |x-1| < 1\}$, and this exact sequence also splits: the subgroup $\mu_{(p)}(F) \subseteq \fro_F^\times$ of roots of unity in $F$ of order prime to $p$ maps isomorphically onto $\fro_F^\times/T_0^\circ$. In other words, we have 
\[ F^\times = \underbrace{T_0^\circ \times \mu_{(p)}(F)}_{\fro_F^\times} \times \varpi^{\bbZ}. \]
Any logarithm on $F^\times$ must be given on $T_0^\circ$ by the usual convergent power series 
\[ z \mapsto \sum_{n = 1}^\infty \frac{(-1)^{n-1}}{n} (z-1)^n \]
in order to invert the exponential map. Moreover, since roots of unity are torsion elements of $F^\times$ and the additive group $\frt_E = E$ has no torsion, any logarithm on $F^\times$ must necessarily annihilate $\mu_{(p)}(F)$. In other words, the $\lambda : T_0 \to \frt_E$ that shows up in the proof above is, in this case, the unique homomorphism $\fro_F^\times \to E$ which is given by the usual power series on $T_0^\circ$ and which annihilates $\mu_{(p)}(F)$. To extend this $\lambda$ to a logarithm on $F^\times$, we are free to choose the image of $\varpi$ to be any element of $E$. In other words, the set of all logarithms on $F^\times$ is in bijection with $E$. This corresponds to the fact that the $\Hom(T/T_0, \frt_E)$ which showed up in the proof above is just $\Hom(\varpi^{\bbZ}, E) = E$ in this example. 
\end{example}

\begin{remark}
The previous example generalizes straightforwardly to a higher-dimensional torus $\bbT = \bbG_m^n$. The reader is invited to study \cref{examples} for some explicit and motivated examples of choosing logarithms on higher dimensional tori. 
\end{remark}

\subsection{Lifting to actions of tori} \label{lifting-tori}

Let $\bbT$ be a split algebraic torus and fix $\log \in \Logs(T)$. 
\begin{lemma} \label{lifting-nilpotent}
Suppose $M$ is a finite dimensional $\frt$-module on which $\frt$ acts nilpotently, ie, such that $\phi_M(\frt) \subseteq \Nil(M)$, where $\Nil(M)$ is the set of nilpotent endomorphisms of $M$ as a vector space. Then the composite 
\[ \begin{tikzcd} T \ar{r}{\log} & \frt_E \ar{r}{\phi_M} & \Nil(M) \ar{r}{\exp} & \GL(M) \end{tikzcd} \]
defines a locally analytic homomorphism $\tilde{\phi}_M : T \to \GL(M)$ and 
\[ \phi_M(x) = \left. \frac{d}{dt} \tilde{\phi}_M(\exp(tx)) \right|_{t = 0} \]
for all $x \in \frt$. Thus $M$ equipped with $\tilde{\phi}_M$ is a representation of $T$ which lifts the original action of $\frt$. Moreover, if $N$ is another finite dimensional $\frt$-module with $\phi_N(\frt) \subseteq \Nil(N)$ and $\sigma : M \to N$ is a homomorphism of $\frt$-modules, then $\sigma$ is also a homomorphism of representations of $T$. 
\end{lemma}

\begin{proof}
It is clear that $\tilde{\phi}$ is a locally analytic homomorphism, so we only need to check that that it differentiates to $\phi$, but this is more or less identical to the proof of \cref{lifting-unipotent-check}. Fix $x \in \frt$. If $t$ is small enough that $\exp$ is defined on $tx$, we have $\log(\exp(tx)) = tx$, which means that 
\[ \tilde{\phi}(\exp(tx)) = \exp(\phi(\log(\exp(tx)))) = \exp(\phi(tx)) = \exp(t\phi(x)), \]
so 
\[ \left. \frac{d}{dt} \tilde{\phi}(\exp(tx)) \right|_{t = 0} = \left. \frac{d}{dt} \exp(t\phi(x)) \right|_{t = 0} = \exp(t\phi(x))\phi(x) \bigg|_{t = 0} = \phi(x). \]
Finally, if $\sigma : M \to N$ is a homomorphism of finite dimensional $\frt$-modules with $\phi_N(\frt) \subseteq \Nil(N)$, then for $t \in T$ and $m \in M$, we have
\[ \begin{aligned} \sigma(\tilde{\phi}_M(t)(m)) &= \sigma( \exp(\phi_M(\log(t)))(m) ) \\
&= \sigma( m + \phi_M(\log(t))(m) + \dotsb ) \\
&= \sigma(m) + \sigma(\phi_M(\log(t))(m)) + \dotsb \\
&= \sigma(m) + \phi_N(\log(t))(\sigma(m)) + \dotsb \\
&= \exp(\phi_N(\log(t))(\sigma(m)) \\
&= \tilde{\phi}_N(t)(\sigma(m)) \end{aligned} \]
which proves that $\sigma$ is also a homomorphism of representations of $T$.  
\end{proof}

We now generalize the above lifting procedure. 

\begin{para}
If $M = M_\lambda$ for a single algebraic $\lambda \in \frt^*_E$, we define $\tilde{\phi}_M : T \to \GL(M)$ by 
\[ \tilde{\phi}_M(t) = \chi_\lambda(t) \exp(\phi_n(\log(t))). \] 
This defines a locally analytic action of $T$ on $M$ lifting the original action of $\frt$. 
\end{para}

\begin{para} \label{lifting-tori-general}
More generally, if $M$ is a split finite dimensional $\frt$-module with algebraic weights, we can lift the action of $\frt$ one generalized weight space at a time. In other words, let $\chi : T \to \GL(M)$ be the map where $\chi(t)$ acts on $M_\lambda$ by multiplication by $\chi_\lambda(t)$. Then defining $\tilde{\phi}_M : T \to \GL(M)$ by  
\[ \tilde{\phi}_M(t) = \chi(t) \exp(\phi_n(\log(t))) \]
gives a lift of the action of $\frt$. We write $\Lift(M, \log)$ to denote $M$ regarded as a representation of $T$ via this procedure.
\end{para}

\begin{para} \label{torus-action-fd}
Since the generalized weight space decomposition is functorial, this lifting construction is also functorial. In other words, $M \mapsto \Lift(M, \log)$ defines a functor
\[ \Lift(-, \log) : \Mod_\frt^{\alg,\fd} \to \Rep_T^\an \]
on the category $\Mod_{\frt}^{\alg,\fd}$ of finite dimensional $\frt$-modules with algebraic weights. 
\end{para}

\begin{para}[Extension to locally finite dimensional modules] \label{torus-action-lfd}
Finally, we extend the lifting functor $\Lift(-,\log)$ from $\Mod_\frt^{\alg,\fd}$ to all of $\Mod_\frt^{\alg}$ in the universal way, using a left Kan extension \cite[chapter X]{MacLaneCats}. Let us tentatively write $\Lift^*(-,\log)$ for the left Kan extension of $\Lift(-,\log)$ along the inclusion $\Mod_\frt^{\alg,\fd} \hookrightarrow \Mod_\frt^{\alg}$ (we will ultimately drop the asterisk from this notation). This left Kan extension can be described explicitly as follows: if $M \in \Mod_\frt^{\alg}$, then  
\begin{equation} \label{lan-definition} 
\Lift^*(M, \log) = \colim_{N \in \cI(M)} \Lift(N, \log) 
\end{equation}
where $\cI(M)$ is the category of $\frt$-module homomorphisms $N \to M$ where $N \in \Mod^{\alg,\fd}_{\frt}$. Observe that the subcategory of finite dimensional $\frt$-submodules of $M$ is cofinal in $\cI(M)$, since any $\frt$-module homomorphism $N \to M$ in $\cI(M)$ factors through the finite dimensional $\frt$-submodule $f(N)$. Thus we may replace $\cI(M)$ with this cofinal subcategory without affecting \cref{lan-definition}. In other words, $\Lift^*(M, \log)$ is just the union of $\Lift(N, \log)$ over all finite dimensional $\frt$-submodules $N \subseteq M$. Now, if $M \in \Mod_{\frp}^{\alg,\fd}$, it is clear that there is a natural isomorphism $\Lift^*(M,\log) = \Lift(M,\log)$. Thus, we may remove the asterisk and define 
\[ \Lift(-,\log) : \Mod_\frt^{\alg} \to \Rep_P^\an \]
to be the left Kan extension of the functor on $\Mod_\frp^{\alg,\fd}$ defined in \cref{torus-action-fd} without introducing any conflict of notation. 
\end{para}

We conclude this subsection with the following observation about how this lifting construction interacts with passage to subtori. 

\begin{lemma} \label{subtorus}
Let $\bbS \subseteq \bbT$ be split algebraic tori and suppose $\log \in \Logs(T)$. Then $\mathrm{log}|_{S} \in \Logs(S)$, and if $M \in \Mod_\frt^\alg$, then 
\[ \Lift(M, \log)|_{S} = \Lift(M|_{\frs}, \mathrm{log}|_S). \]
\end{lemma}

\begin{proof}
Let $S_0$ be the maximal compact subgroup of $S$. Then $\log|_{S_0}$ must be the unique logarithm map on $S_0$ in the sense of \cite[III.6]{B-L}, so $\log(S_0) \subseteq \frs$. Then, for any $s \in S$, there exists a positive integer $n$ such that $s^n \in S_0$, which means that $n\log(s) = \log(s^n) \in \frs$. Since $\frs$ is a divisible abelian group, we conclude that $\log(s) \in \frs$ as well. This shows that $\log(S) \subseteq \frs$, from which it follows that $\log|_S \in \Logs(S)$. The latter statement follows immediately. 
\end{proof}

\subsection{Compatibility of the torus action}

Suppose $\bbG$ is a connected algebraic group and $\bbT$ is a split maximal torus of $\bbG$. Then $\frt$ is a Cartan subalgebra of $\frg$, and we have the root space decomposition 
\[ \frg = \bigoplus_{\alpha \in \frt_E^*} \frg_\alpha. \] 

\begin{lemma} \label{root-space-generalized-weight-space}
Suppose $M$ is a $\frg$-module. Then the root space $\frg_\alpha$ maps the generalized weight space $M_\lambda$ into $M_{\alpha + \lambda}$ for all $\alpha, \lambda \in \frt^*_E$.
\end{lemma}

\begin{proof}
Suppose $x \in \frg, y \in \frt$ and $m \in M$. Then 
\[ \begin{aligned} (\phi(y) - \alpha(y) - \lambda(y))(x \bdot m) &= y \bdot (x \bdot m) - \alpha(y)(x \bdot m) - \lambda(y)(x \bdot m) \\
&= [y,x] \bdot m + x \bdot (y \bdot m) - \alpha(y)(x \bdot m) - \lambda(y)(x \bdot m) \\
&= (\ad(y) - \alpha(y))(x) \bdot m + x \bdot (\phi(y) - \lambda(y))(m)  \end{aligned} \]
By induction, we have 
\[ (\phi(y) - \alpha(y) - \lambda(y))^n(x \bdot m) = \sum_{k = 0}^n \binom{n}{k} (\ad(y) - \alpha(y))^k(x) \bdot (\phi(y) - \lambda(y))^{n-k}(m). \]
Suppose now that $x \in \frg_{\alpha}$ and $m \in M_\lambda$. Choose $n$ such that $(\phi(y) - \lambda(y))^n(m) = 0$. If $1 \leq k \leq n$, then $(\ad(y)-\alpha(y))^k(x) = 0$ since $x \in \frg_\alpha$. If $k = 0$, then $n-k = n$ so \[ (\phi(y)-\lambda(y))^{n-k}(m) = (\phi(y)-\lambda(y))^{n}(m) = 0 \] by our choice of $n$. This means that all of the summands in the summation above vanish. Thus \[ (\phi(y) - \alpha(y) - \lambda(y))^n(x \bdot m) = 0, \] ie, $x \bdot m$ is a generalized weight vector with weight $\alpha + \lambda$. 
\end{proof}

\begin{lemma} \label{nilpotent-part-commutes}
Let $M$ be a $\frg$-module which splits as a $\frt$-module and let $\phi_n$ denote the nilpotent part of $\phi|_{\frt}$. Then 
\[ \phi(x) \circ \phi_n(y) = \phi_n(y) \circ \phi(x) \]
for all $x \in \frp$ and $y \in \frt$.
\end{lemma}

\begin{proof}
Since $\frg = \bigoplus \frg_\alpha$, we may assume without loss of generality that $x \in \frg_{\alpha}$. Furthermore, since $M = \bigoplus M_\lambda$, it is sufficient to show that $\phi(x)$ and $\phi_n(y)$ commute when the domain is restricted to $M_\lambda$. In other words, it is sufficient to show that 
\begin{equation} \label{nilpotent-part-commutes-sufficient} x \bdot \phi_n(y)(m) = \phi_n(x \bdot m) \end{equation}
for all $m \in M_\lambda$. Let $\phi_s$ denote the semisimple part of $\phi|_{\frt_E}$. Since $x \in \frp_{\alpha,E}$, we have $[y,x] = \ad(y)(x) = \alpha(y) x$, so 
\[ \begin{aligned}
\alpha(y) x \bdot m &= [y,x] \bdot m \\
&= \phi(y)(x \bdot m) - x \bdot \phi(y)(m) \\
&= \phi_s(y)(x \bdot m) + \phi_n(y)(x \bdot m) - x \bdot \phi_s(y)(m) - x \bdot \phi_n(y)(m) \\
&= (\alpha + \lambda)(y)(x \bdot m) - x \bdot \lambda(y)m + \phi_n(y)(x \bdot m) - x \bdot \phi_n(y)(m) \\
&= \alpha(y) x \bdot m + \phi_n(y)(x \bdot m) - x \bdot \phi_n(y)(m).
\end{aligned} \]
Here, we have used \cref{root-space-generalized-weight-space} for the equality $\phi_s(x \bdot m) = (\alpha + \lambda)(x \bdot m)$. Subtracting $\alpha(y) x \bdot m$ from both sides of the above equation
\[ \alpha(y) x \bdot m = \alpha(y) x \bdot m + \phi_n(y)(x \bdot m) - x \bdot \phi_n(y)(m) \]
yields precisely \cref{nilpotent-part-commutes-sufficient}. 
\end{proof}

\begin{lemma} \label{compatibility}
Suppose $M$ is a $\frg$-module such that $M|_\frt \in \Mod^{\alg}_\frt$. Fix $\log \in \Logs(T)$ and let $\tilde{\phi} = \tilde{\phi}_{M|_{\frt}} : T \to \GL(M)$ denote the action of $T$ on $\Lift(M|_\frt, \log)$. Then 
\begin{equation} \label{compatibility-statement} 
\Ad(t)(x) \bdot \tilde{\phi}(t)(m) = \tilde{\phi}(t)(x \bdot m)
\end{equation}
for all $t \in T$, $x \in \frg$, and $m \in M$. 
\end{lemma}

\begin{proof}
Since $\frg = \bigoplus \frg_\alpha$ and $M = \bigoplus M_\lambda$, it is enough to prove \cref{compatibility-statement} for $x \in \frg_{\alpha}$ and $m \in M_\lambda$. 
First of all, notice that, since $m \in M_\lambda$, we have 
\[ \tilde{\phi}(t)(m) = \chi_\lambda(t) \exp(\phi_n(\log(t)))(m) \]
where $\phi_n$ is the nilpotent part of $\phi|_{\frt}$. Then we have $\Ad(t)(x) = \chi_\alpha(t)x$ since $x \in \frg_\alpha$, so
\[ \begin{aligned}
\Ad(t)(x) \bdot \tilde{\phi}(t)(m) &= (\chi_\alpha(t)x) \bdot \left( \chi_\lambda(t) \exp(\phi_n(\log(t)))(m) \right) \\
&= \chi_\alpha(t) \chi_\lambda(t) x \bdot \exp(\phi_n(\log(t)))(m)  \\
\end{aligned} \]
On the other hand, note that $x \bdot m \in M_{\alpha + \lambda}$ by \cref{root-space-generalized-weight-space}, so  
\[ \tilde{\phi}(t)(x \bdot m) = \chi_{\alpha+\lambda}(t) \exp(\phi_n(\log(t)))(x \bdot m). \]
Thus \cref{compatibility-statement} is equivalent to the assertion that 
\[ \phi(x) \circ \exp(\phi_n(\log(t))) = \exp(\phi_n(\log(t))) \circ \phi(x), \]
but this follows immediately from \cref{nilpotent-part-commutes}. 
\end{proof}

\begin{para}
The upshot of \cref{compatibility} is that $\Lift(M|_{\frt}, \log)$ is naturally a ``locally analytic $(\frp,T)$-module'' (cf. \cref{gP-module}). 
\end{para}

\subsection{Lifting to actions of bigger groups} \label{lifting-general}

Suppose $\bbP$ is a connected algebraic group and $\bbT$ is a split maximal torus of $\bbP$. 

\begin{definition} \label{alg-unilp-definition}
Let $\Mod_{\frp}^\alg$ denote the category of locally finite dimensional $\frp$-modules $M$ such that $M|_{\frt}$ is split with algebraic weights. Inside this, we have the following subcategories: 
\begin{itemize}
\item $\Mod_{\frp}^{\alg,\fd}$ is the subcategory of finite dimensional $\frp$-modules in $\Mod_\frp^\alg$. 
\item $\Mod_\frp^{\alun}$ is the subcategory of $\frp$-modules in $\Mod_\frp^\alg$ on which $\fru$ acts locally nilpotently (cf. \cref{nilp-definition}), where $\fru$ is the Lie algebra of the unipotent radical $\bbU$ of $\bbP$. 
\item $\Mod_\frp^{\alun,\fd} = \Mod_\frp^{\alun} \cap \Mod_\frp^{\alg,\fd}$. 
\end{itemize}
These are all Serre subcategories of $\Mod_\frp$ by \cref{fd-lfd-serre,alg-serre,nilp-serre}. 
\end{definition}

The goal in this section is to define an action of $P$ on $\Lift(M|_{\frt},\log)$ for any $M \in \Mod_\frp^{\alun}$. We do this incrementally. Until \cref{kan-extension}, we assume that $M \in \Mod_{\frp}^{\alun,\fd}$. 

\begin{para}[Semisimple groups] \label{lifting-semisimple}
Suppose $\bbP$ is semisimple. The Cartan subalgebra $\frt$ consists of semisimple elements in the semisimple Lie algebra $\frp$, and representations of semisimple Lie algebras preserve Jordan decompositions \cite[section 6.4]{humphreys}; thus, $\frt$ acts semisimply on $M$. Since $M|_\frt$ has algebraic weights, the action of $T$ induced by the construction of \cref{lifting-tori} is algebraic; in other words, $M$ is naturally a $(\frp, \bbT)$-module in the sense of \cite[part II, section 1.20]{Jantzen}, so the discussion in \emph{loc.~cit.}~implies that $M$ naturally has an action of the entire group $P$ which lifts the original action of $\frp$. We write $\Lift(M)$ to denote this representation of $P$. 

Note also that, since elements of the root spaces of $\frp$ are $\ad$-nilpotent, they are equal to the nilpotent part of their abstract Jordan decomposition. Therefore, their action on $M$ is nilpotent because $\frp$ is semisimple \cite[section 6.4]{humphreys}. Thus the action of the root subgroups of $P$ must be the action constructed in \cref{lifting-unipotent}. 
\end{para}

\begin{para}[Reductive groups] \label{lifting-reductive}
Suppose $\bbP$ is reductive. Let $\rho : T \to \GL(M)$ denote the action of $T$ on $\Lift(M|_{\frt}, \log)$. Let $\bbP'$ be the derived subgroup of $\bbP$, which is a semisimple group \cite[8.1.6]{springer}. Then $\bbT' = \bbP' \cap \bbT$ is a maximal torus in $\bbP'$ \cite[exercise 27.9]{Humphreys_Groups} \cite[theorem 17.82]{MilneGroups}, so we can lift the action of $P'$ on $M$ as in \cref{lifting-semisimple}. Let $\sigma : P' \to \GL(M)$ denote the action of $P'$ on $\Lift(M|_{\frp'})$. It follows from \cref{subtorus} that $\sigma_{T'} = \rho|_{T'}$. Note that $P'$ is normal in $P$, and we have 
\begin{equation} \label{reductive-compatibility} \rho(t) \circ \sigma(h) = \sigma(h) \circ \rho(t) \end{equation}
for all $t \in T$ and $h \in P'$ (see proof below). This means that $\rho \times \sigma$ defines a homomorphism $T \ltimes P' \to \GL(M)$, where $T \ltimes P'$ is the external semidirect product of $T$ and $P'$ in which $T$ acts on $P'$ by conjugation. But $P$ is generated by $T$ and $P'$ by  \cref{generated-by-torus-and-derived}, so the multiplication map $T \ltimes P' \to P$ is surjective and we have an exact sequence as follows. 
\[ \begin{tikzcd} 1 \ar{r} & T' \ar{r} & T \ltimes P' \ar{d}[swap]{\rho \times \sigma} \ar{r} & P \ar{r} \ar[dotted,bend left]{dl} & 1 \\
& & \GL(M) \end{tikzcd} \]
Since $\rho$ and $\sigma$ agree on $T'$, it follows that $\rho \times \sigma$ factors through $\tilde{\phi} : P \to \GL(M)$ as indicated above. Let $\Lift(M, \log)$ denote this representation of $P$. 
\end{para}

\begin{proof}[Proof of \cref{reductive-compatibility}] 
Let $\phi_n$ denote the nilpotent part of $\phi|_\frt$ and let $\chi : T \to \GL(M)$ represent the action of $T$ where $T$ acts on a generalized weight space $M_\lambda$ by $\chi_\lambda$, as in \cref{lifting-tori-general}. Then
\[ \rho(t) = \chi(t) \exp(\phi_n(\log(t)). \]
By \cref{generated-by-torus-and-root-subgroups}, we know that the group of $F$-points $P'$ is generated by $T'$ and its root subgroups $U_\alpha$ for all $\alpha \in \Phi(P', T)$. Thus, it is sufficient to show \cref{reductive-compatibility} for $h \in T'$ and $h \in U_\alpha$. Suppose first that $h \in T'$. Then $\sigma(h) = \chi(h)$ is just a scalar on each generalized weight space, so \cref{reductive-compatibility} is clear. Next, suppose that $h$ is in a root subgroup of $P'$. As we noted at the end of \cref{lifting-semisimple}, the action of $h$ is then given by 
\[ \sigma(h) = \exp(\phi(\log(h))), \]
as in \cref{lifting-unipotent}. Thus, to show \cref{reductive-compatibility} in this case, it is sufficient to show that $\phi_n(\log(t))$ and $\phi(\log(h))$ commute. This follows from \cref{nilpotent-part-commutes}. 
\end{proof}

\begin{para}[General groups] \label{lifting-general-sub}
Suppose $\bbP$ is any connected algebraic group. By a theorem of Mostow's, $\bbP$ has a Levi decomposition $\bbU \rtimes \bbL$, where $\bbU$ is the unipotent radical of $\bbP$ and $\bbL$ is a reductive subgroup \cite[VIII.4.3]{Hochschild-Groups}. Note that the maximal torus $\bbT$ of $\bbP$ is contained in $\bbL$. 

Let $\rho : U \to \GL(M)$ be the action induced by the construction of \cref{lifting-unipotent}. Let $\log \in \Logs(T)$, and let $\sigma : L \to \GL(M)$ be the action of $L$ on $\Lift(M|_{\frl}, \log)$ induced by the construction of \cref{lifting-reductive} above. To show that $\rho$ and $\sigma$ induce an action of the entire group $P$, it is sufficient to show that 
\begin{equation} \label{lifting-general-compatibility} \rho(huh^{-1}) = \sigma(h) \circ \rho(u) \circ \sigma(h)^{-1}
\end{equation}
for all $u \in U$ and $h \in L$. 
\end{para}

\begin{proof}[Proof of \cref{lifting-general-compatibility}]
Let $\bbU_\alpha$ be the root subgroup of $\bbL$ corresponding to $\alpha \in \Phi(\bbL, \bbT)$, so that $\Lie(\bbU_\alpha) = \frl_\alpha$. The group $L$ is generated by $T$ and $U_\alpha$ for all $\alpha \in \Phi(\bbL, \bbT)$ by \cref{generated-by-torus-and-root-subgroups}. Thus it suffices to assume that $h$ is in $T$ or in $U_\alpha$ for some $\alpha$. Suppose first that $h \in T$. Note that $\rho(u) = \exp(\phi(\log(u)))$ for all $u \in U$. Setting $x = \log(u)$, we see that it is sufficient to show that 
\[ \phi(\Ad(h)(x)) = \sigma(h) \circ \phi(x) \circ \sigma(h)^{-1} \]
for all $x \in \fru$. This is precisely \cref{compatibility}. 

Next, suppose that $h \in U_\alpha$ for some $\alpha$. Let $\bbU \rtimes \bbU_\alpha$ be the algebraic subgroup of $\bbP$ generated by $\bbU$ and $\bbU_\alpha$. Since the class of unipotent groups is stable under extension \cite[corollary 14.7]{MilneGroups}, we know that $\bbU \rtimes \bbU_\alpha$ is also unipotent. Moreover, $\Lie(\bbU \rtimes \bbU_\alpha) = \fru + \frl_\alpha$ acts nilpotently on $M$, by \cref{semidirect-nilpotent} below. Let $\rho_\alpha : U \rtimes U_\alpha \to \GL(M)$ denote the action on $M$ furnished by \cref{lifting-unipotent}. Since $\rho_\alpha|_U$ and $\rho$ are both actions of $U$ on $M$ which lift $\phi|_\fru$, we must have $\rho_\alpha|_U = \rho$. Similarly, since $\sigma|_{U_\alpha}$ and $\rho_\alpha|_{U_\alpha}$ are both actions of the unipotent group $U_\alpha$ on $M$ which lift $\phi|_{\frl_\alpha}$, we must have $\sigma|_{U_{\alpha}} = \rho_{\alpha}|_{U_\alpha}$. Thus 
\[ \rho(huh^{-1}) = \rho_\alpha(huh^{-1}) = \rho_\alpha(h) \circ \rho_\alpha(u) \circ \rho_\alpha(h)^{-1} = \sigma(h) \circ \rho(u) \circ \sigma(h)^{-1}, \]
proving \cref{lifting-general-compatibility} for $h \in U_\alpha$.
\end{proof}

\begin{lemma} \label{semidirect-nilpotent}
For $M \in \Mod_\frp^{\alun,\fd}$ and $\alpha \in \Phi(\bbL,\bbT)$, the Lie subalgebra $\frl_\alpha + \fru \subseteq \frp$ acts nilpotently on $M$.
\end{lemma}

\begin{proof}
Let $x \in \frl_\alpha$ and $y \in \fru$. Because $[\frl_\alpha,\fru] \subseteq \fru$ it follows easily by induction that $x \bdot (\fru^i \bdot M) \subseteq \fru^i \bdot M$, where $\fru^0 \bdot M := M$ and $\fru^{i+1} \bdot M := \fru \bdot  (\fru^i \bdot M)$. We now consider an element $w$ of $U(\frp)$ of the form 
\begin{numequation}\label{word}
x^{e_1}y^{f_1}x^{e_2}y^{f_2} \cdot \ldots \cdot x^{e_\ell}y^{f_\ell},
\end{numequation} 
with the condition that for all $j \in \{1, \ldots, \ell\}$ we have $e_j,f_j \in \{0,1\}$ and $e_j+f_j = 1$. It follows from what we have just seen, that $w \bdot M \subseteq \fru^{\sum_j f_j} \bdot M$. Now suppose that $x^{n+1} \bdot M =0$ and $\fru^{m+1} \bdot M = 0$ for non-negative integers $n$ and $m$. Let $w$ be as in \ref{word} with $\ell = \sum_j e_j+f_j \ge (m+1)(n+1)$. If there are more than $n$ consecutive $e_j =1$ in \ref{word}, then $w \bdot M = 0$. If there are no $n$ consecutive $e_j = 1$ in \ref{word} then there must be least $m+1$ occurrences of $f_j=1$, hence $w \bdot M \in \fru^{m+1} \bdot M = 0$. This shows that $(x+y)^\ell \bdot M = 0$
\end{proof}

\begin{para}
We let $\Lift(M, \log)$ denote this locally analytic representation of $P$. We have thus defined a functor \[ \Lift(-, \log) : \Mod_{\frp}^{\alun,\fd} \to \Rep_P^\an \] for any connected algebraic group $\bbP$. 
\end{para}

\begin{para} \label{kan-extension}
Finally, we drop the finite dimensionality assumption: we extend the functor $\Lift(-,\log)$ from $\Mod_\frp^{\alun,\fd}$ to all of $\Mod_\frp^{\alun}$ using a left Kan extension \cite[chapter X]{MacLaneCats}. By exactly the same argument described in \cref{torus-action-lfd}, it does not introduce a conflict of notation to denote the left Kan extension of the functor $\Lift(-,\log) : \Mod_\frp^{\alun,\fd} \to \Rep_P^\an$ along the inclusion $\Mod_\frp^{\alun,\fd} \hookrightarrow \Mod_\frp^{\alun}$ 
by 
\[ \Lift(-,\log) : \Mod_\frp^{\alun} \to \Rep_P^\an. \]
Moreover, exactly as in \cref{torus-action-lfd}, this left Kan extension is computing by taking a colimit over finite dimensional $\frp$-submodules of $M$, each of which is in $\Mod_{\frp}^{\alun,\fd}$. 
\end{para}

\subsection{Changing the logarithm}

Suppose $\log, \log' \in \Logs(T)$. The exact sequence \eqref{torsor-exact-sequence} shows that the difference $\epsilon = \log' - \log$ is a homomorphism $T \to \frt$ which annihilates $T_0$. In other words, it is locally constant. 

Let $M$ be a finite dimensional $\frt$-module such that $\phi(\frt) \subseteq \Nil(M)$, and let $\tilde{\phi}$ and $\tilde{\phi}'$ denote the homomorphisms $T \to \GL(M)$ representing the actions of $T$ on $\Lift(M,\log)$ and $\Lift(M,\log')$, respectively. Then 
\[ \tilde{\phi}'(t) = \exp(\phi(\log'(t))) = \exp(\phi(\log(t) + \epsilon(t))) = \exp(\phi(\log(t)) + \phi(\epsilon(t))). \]
Observe that 
\[ [\phi(\log(t)), \phi(\epsilon(t))] = \phi([\log(t), \epsilon(t)]) = 0 \]
since $\frt$ is abelian, so
\[ \tilde{\phi}'(t) = \exp(\phi(\log(t))) \circ \exp(\phi(\epsilon(t))). \]
We set $\tilde{\epsilon}(t) = \exp(\phi(\epsilon(t)))$. This is a locally constant, i.e., smooth, representation $T \to \GL(M)$, and $\tilde{\phi}'(t) = \tilde{\phi}(t) \circ \tilde{\epsilon}(t)$.

%% file: categories-o-infinity.tex
\section{Categories \texorpdfstring{$\cO^{\frp,\infty}$}{O p infinity} and \texorpdfstring{$\cO^{\frp,\infty}_\alg$}{O p infinity alg}} \label{cat-o-infinity}

Let $(\bbG, \bbT)$ be connected a split reductive algebraic group over $F$ and let $\bbP$ be a parabolic subgroup containing $\bbT$. 

\subsection{Definition and basic properties}

\begin{definition}
Let $\cO^{\frp,\infty}$ be the full subcategory of $\Mod_\frg$ consisting of finitely generated $\frg$-modules $M$ such that $M|_{\frp}$ is locally finite dimensional and $M|_{\frt}$ is split (cf. \cref{definition-split}). 
\end{definition}

\begin{proposition} \label{o-infinity-serre}
$\cO^{\frp,\infty}$ is a Serre subcategory of $\Mod_\frg$. 
\end{proposition}

\begin{proof}
Suppose 
\[ \begin{tikzcd} 0 \ar{r} & M' \ar{r} & M \ar{r} & M'' \ar{r} & 0 \end{tikzcd} \]
is an exact sequence of $\frg$-modules. It is clear that if $M'$ and $M''$ are finitely generated, then $M$ is as well; it is also clear that if $M$ is finitely generated, so is $M''$. Since $U(\frg)$ is noetherian, we also know that $M$ being finitely generated implies that $M'$ is. Also, it follows from \cref{fd-lfd-serre} that $M|_{\frp}$ is locally finite dimensional if and only if $M'|_{\frp}$ and $M''|_{\frp}$ are. Finally, it follows from \cref{split-serre} that $M|_{\frt}$ is split if and only if $M'|_{\frt}$ and $M''|_{\frt}$ are. 
\end{proof}

\begin{definition}
Let $\cO^{\frp,\infty}_{\alg}$ be the full subcategory of $\cO^{\frp,\infty}$ consisting of modules $M$ such that $M|_{\frt} \in \Mod_{\frt}^{\alg}$.
\end{definition}

\begin{proposition} \label{o-infinity-alg-serre}
$\cO^{\frp,\infty}_\alg$ is a Serre subcategory of $\Mod_\frg$. 
\end{proposition}

\begin{proof}
This follows from \cref{o-infinity-serre,alg-serre}. 
\end{proof}

\subsection{Relation to categories \texorpdfstring{$\cO^\frp$}{O p} and \texorpdfstring{$\cO^\frp_\alg$}{O p alg}}

\begin{definition}
For any positive integer $n$, let $\cO^{\frp,n}$ be the full subcategory of $\Mod_\frg$ consisting of finitely generated $\frg$-modules $M$ such that $M|_\frp$ is locally finite dimensional and $M = \bigoplus_{\lambda \in \frt_E^*} M_{\lambda,n}$, where \[ M_{\lambda,n} = \{m \in M : (\phi(x) - \lambda(x))^n(m) = 0 \text{ for all } x \in \frt_E \}. \]
Observe that $M_{\lambda,n} \subseteq M_{\lambda,n+1} \subseteq M_\lambda$ and that $\cO^{\frp,n} \subseteq \cO^{\frp,n+1} \subseteq \cO^{\frp,\infty}$ for all $n$. We define $\cO^{\frp} := \cO^{\frp,1}$. 
\end{definition}

\begin{lemma} \label{filtering-cat-o-p-infinity}
$\cO^{\frp,\infty} = \bigcup_{n \geq 1} \cO^{\frp,n}$. 
\end{lemma}

\begin{proof}
Suppose $M \in \cO^{\frp,\infty}$. Since $M$ is a finitely generated $\frg$-module and since $M = \bigoplus M_\lambda$, we can choose a finite set of generators $m_1, \dotsc, m_r$ such that $m_i \in M_{\lambda_i}$ for all $i$. Since $m_i \in M_{\lambda_i}$, there exists an integer $n_i$ such that \[ (\phi(x)-\lambda(x))^{n_i}(m_i) = 0 \] for all $x \in \frt$. Set $n = \max n_i$, so that $m_i \in M_{\lambda_i, n}$ for all $i$. It then follows from \cref{root-space-generalized-weight-space} that $M = \bigoplus_{\lambda \in \frt_E^*} M_{\lambda,n}$, proving that $M \in \cO^{\frp,n}$. 
\end{proof}

\begin{proposition}[{\cite[2.1.2.5]{SoergelDiploma}}] \label{composition-factors-in-cat-o}
A $\frg$-module $M$ is in category $\cO^{\frp,\infty}$ if and only if it has a finite filtration whose successive quotients are in $\cO^\frp$. 
\end{proposition}

\begin{proof}
If $M$ has a finite filtration whose successive quotients are in $\cO^\frp$, the fact that $M \in \cO^{\frp,\infty}$ follows by induction on the length of the filtration using \cref{o-infinity-serre}. Conversely, suppose $M \in \cO^{\frp,\infty}$. Then there exists an $n$ such that $M \in \cO^{\frp,n}$ by \cref{filtering-cat-o-p-infinity}. Let 
\[ M_k = \bigoplus_{\lambda \in \frt_E^*} M_{\lambda,k}, \]
and observe that each $M_k$ is a $\frg$-submodule of $M$ by \cref{root-space-generalized-weight-space}. In other words, we have a filtration \[ M = M_n \supseteq M_{n-1} \supseteq \dotsb \supseteq M_1 \supseteq M_0 = 0 \]
of $M$ by $\frg$-submodules. It is clear from definitions that $Q_i := M_{i}/M_{i-1} \in \cO^{\frp}$. 
\end{proof}

\begin{corollary}
Simple objects in $\cO^{\frp,\infty}$ are in $\cO^{\frp}$. \qed
\end{corollary}

\begin{corollary}
Every object in $\cO^{\frp,\infty}$ has finite length. 
\end{corollary}

\begin{proof}
This follows from \cref{composition-factors-in-cat-o} plus the fact that objects in category $\cO^{\frp}$ have finite length. 
\end{proof}

\begin{definition}
Let $\cO^\frp_\alg = \cO^\frp \cap \cO^{\frp,\infty}_\alg$. 
\end{definition}

\begin{proposition}
A $\frg$-module $M$ is in the category $\cO^{\frp,\infty}_\alg$ if and only if it has a finite filtration whose successive quotients are in $\cO^\frp_\alg$. 
\end{proposition}

\begin{proof}
One direction follows from \cref{o-infinity-alg-serre}. Conversely, if $M$ is in category $\cO^{\frp,\infty}_\alg$, we know from \cref{composition-factors-in-cat-o} that $M$ has a filtration \[ M = M_n \supseteq M_{n-1} \supseteq \dotsb \supseteq M_0 = 0 \]
where $Q_i := M_{i}/M_{i-1} \in \cO^{\frp}$. Since $M \in \cO^{\frp,\infty}_\alg$, we know that $M_i \in \cO^{\frp,\infty}_\alg$ for all $i$ by \cref{o-infinity-alg-serre}. Furthermore, we have exact sequences 
\[ \begin{tikzcd} 0 \ar{r} & M_{i-1} \ar{r} & M_{i} \ar{r} & Q_i \ar{r} & 0, \end{tikzcd} \] 
so $Q_i \in \cO^{\frp,\infty}_\alg$ again by \cref{o-infinity-alg-serre}. Thus $Q_i \in \cO^{\frp,\infty}_\alg \cap \cO^{\frp} = \cO^{\frp}_\alg$. 
\end{proof}

\subsection{Lifting from category \texorpdfstring{$\cO^{\frp,\infty}_\alg$}{O p infinity alg}}

Suppose $M \in \cO^{\frp,\infty}_\alg$ and fix $\log \in \Logs(T)$. 

\begin{para} \label{lift-oinfty}
As $\cO^{\frp,\infty}_\alg$ is the extension closure of $\cO^{\frp}_\alg$, and the restriction to $\frp$ of every object in $\cO^{\frp}_\alg$ is in the Serre subcategory $\Mod_{\frp}^{\alun}$ (cf. \cref{alg-unilp-definition}), it follows that $M|_{\frp} \in \Mod_\frp^{\alun}$. Thus we can form the lift \[ \Lift(M, \log) := \Lift(M|_{\frp}, \log). \]
This is a locally analytic representation of $P$. Also, $\Lift(M, \log)$ is a $\frg$-module since $M = \Lift(M, \log)$ as sets. Since the action of $P$ on $\Lift(M, \log)$ lifts the original action of $\frp$, the two actions of $\frp$ agree. The proposition below shows that $\Lift(M, \log)$ is a locally analytic $(\frg,P)$-module in the sense of \cref{gP-module}. Thus it is naturally a $D(\frg,P)$-module by \cref{gP-module-DgP-module}. 
\end{para}

\begin{proposition}
Let $\tilde{\phi} :  P \to \GL(M)$ denote the action of $P$ on $\Lift(M|_{\frp}, \log)$. Then
\[ \Ad(h)(x) \bdot \tilde{\phi}(h)(m) = \tilde{\phi}(h)(x \bdot m) \]
for $m \in M$, $x \in \frg$, and $h \in P$.\footnote{Since $\tilde{\phi}$ lifts the action of $\frp$ on $M$, we already know this for $x \in \frp$. We only need to prove it for $x \in \frg_\beta$ for roots $\beta$ which are not roots of $\frp$, but it does not help the proof to assume this.}
\end{proposition}

\begin{proof}
Let $\bbU$ be the unipotent radical of $\bbP$, $\bbL$ be a Levi subgroup of $\bbP$, and let $\bbU_\alpha$ be the root subgroup of $\bbG$ corresponding to $\alpha \in \Phi(\bbG,\bbT)$. Then $P$ is generated by $U$ and $L$, and $L$ in turn is generated by $T$ and $U_\alpha$ for $\alpha \in \Phi(\bbL,\bbT)$ by \cref{generated-by-torus-and-root-subgroups}. Moreover, $U$ is generated by $U_\alpha$ for $\alpha \in \Phi(\bbG,\bbT)^+ \setminus \Phi(\bbL,\bbT)$. Thus $P$ is generated by the subgroups $T$ and $U_\alpha$ for $\alpha \in \Phi(\bbL,\bbT) \cup \Phi(\bbG,\bbT)^+$. Said differently, $P$ is generated by $T$ and all of the $U_\alpha$ for roots $\alpha$ such that $\bbU_\alpha \subseteq \bbP$. It thus suffices to show the desired identity for $h \in T$ and for $h \in U_\alpha$ for such $\alpha$. For $h \in T$, this is \cref{compatibility}.

Suppose $h \in U_\alpha$ for $\alpha$ such that $\bbU_\alpha \subseteq \bbP$. Let $y = \log(h)$ be the image of $h$ under the isomorphism $\log : U_\alpha \to \frg_\alpha$ so that $\Ad(h) = \exp(\ad(y))$ on $\frg$. Note that $\Lie(U_\alpha) = \frg_\alpha$ must act locally nilpotently on $M$ for all of these $\alpha$. Indeed, if $\bbU_\alpha \subseteq \bbU$, then $\frg_\alpha \subseteq \fru$ and local nilpotence of the action of $\frg_\alpha$ on $M$ follows from the fact that $\fru$ acts locally nilpotently on $M$ as we observed in \cref{lift-oinfty}; otherwise, if $\bbU_\alpha \not\subseteq \bbU$, then $\bbU_\alpha \subseteq \bbL$, which means $\bbU_\alpha$ is contained in the derived subgroup $\bbL'$ of $\bbL$ (cf. proof of \cref{generated-by-torus-and-derived}), and then local nilpotence of the action of $\frg_\alpha$ on $M$ follows from our observations in \cref{lifting-semisimple}. Thus $\tilde{\phi}(h) = \exp(\phi(y))$ as endomorphisms of $M$, and we want to show that 
\[ \exp(\ad(y))(x) \bdot \exp(\phi(y))(m) = \exp(\phi(y))(x \bdot m). \]
One first shows by induction that 
\[ y^nx = \sum_{k = 0}^n \binom{n}{k} \ad(y)^k(x) y^{n-k} \]
in $U(\frg)$ for all non-negative integers $n$, from which it follows that  
\[ \begin{aligned} \exp(\phi(y))(x \bdot m) &= \sum_{n = 0}^\infty \sum_{k = 0}^n \binom{n}{k} \frac{\ad(y)^k(x) y^{n-k}}{n!} \bdot m \\
&= \sum_{n = 0}^\infty \sum_{k = 0}^n \frac{\ad(y)^k(x)}{k!} \frac{y^{n-k}}{(n-k)!} \bdot m. \end{aligned} \]
On the other hand, we have 
\[ \exp(\ad(y))(x) \bdot \exp(\phi(y))(m) = \left( \sum_{i = 0}^\infty \frac{\ad(y)^i(x)}{i!} \right) \bdot \left( \sum_{j = 0}^\infty \frac{y^j}{j!} \bdot m \right)  \]
and the right-hand sides of these two equations are evidently equal. 
\end{proof}

%% file: functor.tex
\section{Globalization functor}

Let $(\bbG, \bbT)$ be a connected split reductive algebraic group over $F$ and let $\bbP$ be a parabolic subgroup containing $\bbT$. Also, fix a logarithm $\log \in \Logs(T)$. 

\subsection{Setup and preliminaries}

Let $M \in \cO^{\frp,\infty}_\alg$ and let $V$ be a smooth strongly admissible representation of $P$ over $E$.\footnotemark

\footnotetext{For example, if $V$ is a smooth admissible representation of $P$ of finite length, it is strongly admissible. Indeed, the unipotent radical of $P$ acts trivially on $V$ \cite[lemme 13.2.3]{Boyer99}, so $V$ is a smooth representation of the Levi quotient of finite length. These are strongly admissible by \cite[proposition 2.2]{ST01a}.}

\begin{para} \label{dual-smooth}
We endow $V$ with its finest locally convex topology. Then the $D(P)$-module $V'$ is finitely generated as a $D(P_0)$-module \cite[p. 453]{ST02b}, where $P_0$ is a compact open subgroup of $P$, and it is annihilated by $\frp$ \cite[proposition 2.1]{ST01a}. In other words, $V'$ is a finitely generated module over the quotient $D^\infty(P_0)$ of $D(P_0)$ by the two-sided ideal generated by $\frp$. 
\end{para}

\begin{para} \label{dual-smooth-fp}
In fact, $V'$ is even finitely presented over $D^\infty(P_0)$ (or, equivalently, over $D(P_0)$), where $P_0$ is any compact open subgroup of $P$. Note that smooth strongly admissibile representations of $P$ can be characterized as those whose restrictions to $P_0$ are subrepresentations of $C^\infty(P_0)^{\oplus m}$ for some $m$ \cite[section 2]{ST01b}. If $V$ is a  smooth strongly admissible representation of $P$, we can set up a short exact sequence 
\[ \begin{tikzcd} 0 \ar{r} & V \ar{r} & C^\infty(P_0)^{\oplus m} \ar{r} & W \ar{r} & 0 \end{tikzcd} \]
where $W$ is the cokernel of the inclusion. But $C^\infty(P_0)^{\oplus m}$ is semisimple, so $W$ is itself a subrepresentation of $C^\infty(P_0)^{\oplus m}$, so it is also strongly admissible. Dualizing the above sequence yields an exact sequence
\[ \begin{tikzcd} 0 \ar{r} & W' \ar{r} & D^\infty(P_0)^{\oplus m} \ar{r} & V' \ar{r} & 0. \end{tikzcd} \]
Since $W$ is a smooth strongly admissible representation of $P_0$, its dual $W'$ is also finitely generated over $D^\infty(P_0)$, so we conclude that $V'$ is finitely presented over $D(P_0)$. 
\end{para}

\begin{para} \label{DgP-action-tensor-product}
Since $V'$ is a $D^\infty(P)$-module, we can use the map from \cref{quotients-isomorphic} to regard $V'$ as a $D(\frg,P)$-module on which $\frg$ acts trivially. Then, applying the construction of \cref{tensor-DgP-module-construction}, we see that $\Lift(M,\log) \otimes_E V'$ is naturally a $D(\frg,P)$-module. Since the action of $\frg$ is trivial on $V'$, the action of $x \in \frg$ is given explicitly by  
\[  x \bdot (m \otimes \lambda) = (x \bdot m) \otimes \lambda. \]
\end{para}

\begin{lemma} \label{finite-presentation-FMV-lemma}
If $W$ is any $\frp$-module, and if $X$ is a trivial $\frp$-module, then there is a natural $\frg$-linear isomorphism 
\[ (U(\frg) \otimes_{U(\frp)} W) \otimes_E X = U(\frg) \otimes_{U(\frp)} (W \otimes_E X), \]
where we regard the left-hand side as a $\frg$-module by extending the trivial action of $\frp$ on $X$ to a trivial action of $\frg$. 
\end{lemma}

\begin{proof}
There is a natural $\frp$-linear map 
\[ \begin{tikzcd} W \otimes_E X \ar{r} &  (U(\frg) \otimes_{U(\frp)} W) \otimes_E X \end{tikzcd} \] 
given by $w \otimes x \mapsto (1 \otimes w) \otimes x$, and this induces a natural $\frg$-linear map 
\[ \begin{tikzcd} U(\frg) \otimes_{U(\frp)} (W \otimes_E X) \ar{r} & (U(\frg) \otimes_{U(\frp)} W) \otimes_E X. \end{tikzcd} \]
We know that $U(\frg)$ is free as a left $U(\frp)$-module with basis $U(\fru^-)$, so both sides are compatibly isomorphic to $U(\fru^-) \otimes_E W \otimes_E X$ as vector spaces. Thus the above map is an isomorphism. 
\end{proof}

\begin{proposition} \label{finite-presentation-FMV}
Suppose $M \in \cO^{\frp,\infty}_\alg$ and $V$ is a smooth strongly admissible representation of $P$. If $P_0$ is a compact open subgroup of $P$, then $\Lift(M, \log) \otimes_E V'$ is finitely presented as a $D(\frg,P_0)$-module.  
\end{proposition}

\begin{proof} 
To ease notation, set $\tilde{M} = \Lift(M, \log)$. Let $W$ be a finite dimensional $\frp$-submodule of $M$ which generates $M$ as a $\frg$-module, and set $\tilde{W} = \Lift(W, \log)$. The locally analytic $P$-representation $\tilde{W}$ is naturally a $D(P)$-module \cite[proposition 3.2 and the sentence before lemma 3.1]{ST02b}. Moreover, since $W$ generates $M$ as a $\frg$-module, the natural map $D(\frg,P_0) \otimes_{D(P_0)} \tilde{W} \to \tilde{M}$ is surjective. Let $Z$ be its kernel, so that we have an exact sequence
\begin{equation} \label{presentation-lift-M} \begin{tikzcd} 0 \ar{r} & Z \ar{r} & D(\frg,P_0) \otimes_{D(P_0)} \tilde{W} \ar{r} & \tilde{M} \ar{r} & 0. \end{tikzcd} \end{equation}
Observe that 
\[ D(\frg,P_0) \otimes_{D(P_0)} \tilde{W} = (U(\frg) \otimes_{U(\frp)} D(P_0)) \otimes_{D(P_0)} \tilde{W} = U(\frg) \otimes_{U(\frp)} W \]
where we have used the fact that $D(\frg,P_0) = U(\frg) \otimes_{U(\frp)} D(P_0)$ \cite[sentence after lemma 4.1]{ScSt} and the fact that $\tilde{W} = W$ as a $\frp$-module. Since $W$ is finite dimensional, we see that $U(\frg) \otimes_{U(\frp)} W$ is a finitely generated $\frg$-module. Since $\tilde{M} = M$ is also a finitely generated $\frg$-module, and since $U(\frg)$ is noetherian, we see that $Z$ is also finitely generated as a $\frg$-module. 

Let us tensor \eqref{presentation-lift-M} with $V'$ over $E$. 
\begin{equation} \label{presentation-lift-m-tensor} \begin{tikzcd} 0 \ar{r} & Z \otimes V' \ar{r} & (D(\frg,P_0) \otimes_{D(P_0)} \tilde{W}) \otimes_E V' \ar{r} & \tilde{M} \otimes_E V' \ar{r} & 0 \end{tikzcd} \end{equation}
Let us first show that the term in the middle is naturally a $D(\frg,P_0)$-module in such a way that the map on the right is $D(\frg,P_0)$-linear. Using \cite[sentence after lemma 4.1]{ScSt} as well as  \cref{finite-presentation-FMV-lemma}, note that we have a natural isomorphism 
\[ \begin{aligned} 
(D(\frg,P_0) \otimes_{D(P_0)} \tilde{W}) \otimes_E V' &= (U(\frg) \otimes_{U(\frp)} W) \otimes V' \\ 
&= U(\frg) \otimes_{U(\frp)} (W \otimes V') \\
&= D(\frg,P_0) \otimes_{D(P_0)} (\tilde{W} \otimes_E V'), 
\end{aligned} \]
which shows that the middle term is in fact a $D(\frg,P_0)$-module. Using this natural identification, the fact that the map on the right of \eqref{presentation-lift-m-tensor} is $D(\frg, P_0)$-linear follows from the fact that the map $D(\frg,P_0) \otimes_{D(P_0)} \tilde{W} \to \tilde{M}$ is $D(\frg, P_0)$-linear. It follows that $Z \otimes V'$ is naturally a $D(\frg,P_0)$-module as well. 

Note that $V'$ is a finitely generated $D^\infty(P_0)$-module (cf. \cref{dual-smooth}), and $\tilde{W}$ is a finite dimensional locally analytic representation of $P_0$, so $\tilde{W} \otimes_E V'$ is a finitely generated $D(P_0)$-module by \cref{tensor-product-fd}. Thus the middle term of \eqref{presentation-lift-m-tensor} is a finitely generated $D(\frg,P_0)$-module. 

We now claim that $Z \otimes_E V'$ is also finitely generated. Because $Z$ too is in $\cO^{\frp, \infty}_\alg$, there is a finite-dimensional $\frp$-submodule $Z_1 \sub Z$ which generates $Z$ as $U(\frg)$-module. Then, given any $\lambda \in V'$ and $w \in Z$ there are $w_1, \ldots, w_n \in Z_1$ and $u_1, \ldots, u_n \in U(\frg)$ such that $w = \sum_i u_i.w_i$, and hence $w \otimes \lambda  = \sum_i u_i.(w_i \otimes \lambda)$. Hence $Z \otimes_E V'$ is generated by $Z_1 \otimes_E V'$ as a $U(\frg)$-module. The map $D(\frg,P_0) \otimes_{D(P_0)} (Z_1 \otimes V') \ra Z \otimes_E V'$ is hence surjective, because $Z_1 \otimes_E V'$ is finitely generated as a $D(P_0)$-module by \cref{tensor-product-fd}.
\end{proof} 

\subsection{Definition and properties}

We continue to assume that $M \in \cO^{\frp,\infty}_\alg$ and that $V$ is a smooth strongly admissible representation of $P$ over $E$. As we noted in \cref{DgP-action-tensor-product}, $\Lift(M, \log) \otimes_E V'$ is naturally a $D(\frg,P)$-module.  

\begin{definition}
We define 
\[ \vF^G_P(M,V) = D(G) \otimes_{D(\frg,P)} (\Lift(M, \log) \otimes_E V'). \]
When $G$ and $P$ can be understood from context, we drop the superscript and subscript and simply write $\vF(M,V)$ instead. Also, when $V = E$ is the trivial representation of $P$, we simply write $\vF(M)$.
\end{definition}

\begin{para}
Let $G_0$ be a maximal compact subgroup of $G$, and let $P_0 = G_0 \cap P$. Observe that
\[ \vF(M,V) = D(G_0) \otimes_{D(\frg,P_0)} (\Lift(M, \log) \otimes_E V'), \]
as $D(G_0)$-modules \cite[lemma 4.2]{ScSt}. We use this observation repeatedly below.
\end{para} 

\begin{theorem} \label{coadmissibility}
If $G_0$ is a maximal compact subgroup of $G$, then $\vF(M,V)$ is a finitely presented $D(G_0)$-module. In particular, it is a coadmissible $D(G)$-module. 
\end{theorem}

\begin{proof}
Let $P_0 = G_0 \cap P$. Since $\Lift(M, \log) \otimes_E V'$ is finitely presented as a $D(\frg,P_0)$-module by \cref{finite-presentation-FMV}, the result follows immediately. 
\end{proof}

\begin{theorem} \label{exactness}
The functor $(M,V) \mapsto \vF(M,V)$ is exact in each  argument. 
\end{theorem}

\begin{proof}
Certainly $(M,V) \mapsto \Lift(M,\log) \otimes_E V'$ is exact in each argument. Let $G_0$ be a maximal compact subgroup of $G$, and let $P_0 = G_0 \cap P$. Since $\Lift(M,\log) \otimes_E V'$ is a finitely presented $D(\frg,P_0)$-module, it follows from \cref{exactness-fp} that \[ (M,V) \mapsto D(G_0) \otimes_{D(\frg,P_0)} (\Lift(M,\log) \otimes_E V') = \vF(M,V) \]
is also exact in each argument. 
\end{proof}

\subsection{Change of parabolic}

Let $Q \supseteq P$ be a parabolic subgroup, and let $V$ be a smooth strongly admissible representation of $P$. We let $i(V) = \ind_P^Q(V)$ denote the smooth induction of $V$ to $Q$. 

\begin{lemma}
$i(V)$ is a smooth strongly admissible representation of $Q$. 
\end{lemma}

\begin{proof}
Let $Q_0$ be a compact open subgroup of $Q$ and let $P_0 \subseteq Q_0 \cap P$ be a compact open subgroup. By the characterization of \cite[section 2]{ST01b}, we see that $V$ restricted to $P_0$ is a subrepresentation of $C^\infty(P_0)^{\oplus m}$ for some $m$ \cite[section 2]{ST01b}. Since smooth induction is exact, $\ind_{P_0}^{Q_0}(V)$ is a subrepresentation of 
\[ \ind_{P_0}^{Q_0}(C^\infty(P_0)^{\oplus m}) = C^\infty(Q_0)^{\oplus m} \,, \]
and $\ind_{P_0}^{Q_0}(V)$ is thus strongly admissible again by the characterization of \cite[section 2]{ST01b}. But 
\[ i(V)|_{Q_0} = \bigoplus_{g \in Q_0 \backslash Q/P} \ind^{Q_0}_{Q_0 \cap gPg^{-1}}({}^g V). \] Because $V$ is a strongly admissible as representation of the compact open subgroup $g^{-1}Q_0g \cap P$, so is ${}^g V$ as a representation of $Q_0 \cap gPg^{-1}$. Therefore, $\ind_{Q_0 \cap gPg^{-1}}^{Q_0}({}^g V)$ is strongly admissible by the argument above. Since $P$ is parabolic, the double quotient $Q_0 \backslash Q/P$ is finite, so $i(V)$ is a finite direct sum of strongly admissible representations and is therefore itself smooth strongly admissible. 
\end{proof}

\begin{lemma} \label{smooth-induction-module}
Let $D^\infty(Q)$ denote the quotient of $D(Q)$ by the two-sided ideal $I(\frq)$ generated by $\frq$. There is an isomorphism of $D(Q)$-modules 
\[ D^\infty(Q) \otimes_{D(P)} V' = i(V)' \]
which is natural in the smooth strongly admissible representation $V$. 
\end{lemma}

\begin{proof}
We first claim that \[ D(Q) \otimes_{D(P)} V' = \Ind^Q_P(V)' \]
where $\Ind_P^Q(V)$ is the locally analytic induction. To see this, observe first that $D(Q) \otimes_{D(P)} V'$ is coadmissible. Indeed, suppose $Q_0$ is a compact open subgroup of $Q$ and $P_0 = Q_0 \cap P$. Then $D(Q) \otimes_{D(P)} V' = D(Q_0) \otimes_{D(P_0)} V'$ as $D(P_0)$-modules \cite[lemma 6.1(ii)]{ST05}, and $V'$ is a finitely presented $D(P_0)$-module by \cref{dual-smooth-fp}, so $D(Q_0) \otimes_{D(P_0)} V'$ is a finitely presented $D(Q_0)$-module, which in turn means that $D(Q) \otimes_{D(P)} V'$ is coadmissible.

We now relate $D(Q) \otimes_{D(P)} V'$ to what is denoted $D(Q) \, \widetilde{\otimes}_{D(P)}\, V'$ in the notation of \cite[section 2]{KohlhaaseII}. Note that, by \cite[eq. (53)]{KohlhaaseII}, there an isomorphism of topological vector spaces
\[ D(Q) \,\widetilde{\otimes}_{D(P)}\, V' \simeq D(Q/P) \, \hat{\otimes}_{E,i} V' \]
which depends on the choice of a locally analytic section $Q/P \ra Q$ of the projection map $Q \ra Q/P$. Since $P$ is parabolic, the quotient $Q/P$ is compact, so $D(Q/P)$ is Fr\'echet; thus $D(Q) \,\widetilde{\otimes}_{D(P)}\, V'$ is Fr\'echet as well. Since $D(Q) \otimes_{D(P)} V'$ is coadmissible and therefore complete, there is a natural continuous map $\alpha : D(Q) \,\widetilde{\otimes}_{D(P)}\, V' \to D(Q) \otimes_{D(P)} V'$. By universal property of $D(Q) \otimes_{D(P)} V'$, there is also a natural map $\beta : D(Q) \otimes_{D(P)} V' \to D(Q) \,\widetilde{\otimes}_{D(P)}\, V'$, and it is clear that $\alpha \circ \beta = \id$. This means that $\alpha$ is surjective; since its domain $D(Q) \,\widetilde{\otimes}_{D(P)}\, V'$ is Fr\'echet, we see that $\alpha$ is an open map by the open mapping theorem \cite[proposition 8.6]{NFA}. But then $\beta$ must continuous as well, and it is clear that $\beta \circ \alpha$ is the identity when restricted to $\im(\alpha)$, which is dense in $D(Q) \,\widetilde{\otimes}_{D(P)}\, V'$, which means that $\beta \circ \alpha = 1$ as well. In other words, we have shown that $D(Q) \otimes_{D(P)} V' = D(Q) \widetilde{\otimes}_{D(P)} V'$. 

Observe that 
\[ (D(Q) \otimes_{D(P)} V')' = \Ind_P^Q(V) \]
using \cite[proposition 5.3 and remark 5.4]{KohlhaaseII} together with  reflexivity of $V$. Now note that $D(Q) \otimes_{D(P)} V'$ is also reflexive \cite[proposition 5.3 and theorem 3.1]{KohlhaaseII}, so we can dualize both sides of the above isomorphism to get
\[ D(Q) \otimes_{D(P)} V' = \Ind_P^Q(V)'. \]
Finally, observe that the smooth induction $i(V) = \ind_P^Q(V)$ is the subspace of vectors in the locally analytic induction $\Ind^Q_P(V)$ that are annihilated by $\frq$. Thus 
\[ i(V)' = (D(Q) \otimes_{D(P)} V')/I(\frq) (D(Q) \otimes_{D(P)} V') = D^{\infty}(Q) \otimes_{D(P)} V'. \qedhere \]
\end{proof}

\begin{theorem}\label{PQthm}
Suppose $M \in \cO^{\frq,\infty}_\alg$ and $V$ is a strongly admissible smooth representation of $P$. Then there is an isomorphism of $D(G)$-modules 
\[ \vcF^G_P(M,V) = \vcF^G_Q(M,i(V)) \]
which is natural in both $M$ and $V$. 
\end{theorem} 

\begin{proof}
To ease notation, let us write $\tilde{M} = \Lift(M, \log)$. It is sufficient to prove that
\begin{equation} \label{pq-sufficient}
D(\frg,Q) \otimes_{D(\frg,P)} \left( \tilde{M} \otimes_E V' \right) = \tilde{M} \otimes_E i(V)' \end{equation}
since then applying $D(G) \otimes_{D(\frg,Q)} -$ yields the desired isomorphism. Observe that
\[ \begin{aligned} i(V)' &= D^\infty(Q) \otimes_{D(P)} V' \\
&= D^\infty(Q) \otimes_{D^\infty(P)} V' \\
&= D(\frg,Q)/J(\frg) \otimes_{D(\frg,P)/J(\frg)} V' \\ 
&= D(\frg,Q)/J(\frg) \otimes_{D(\frg,P)} V' \\
&= D(\frg,Q) \otimes_{D(\frg,P)} V',
\end{aligned} \]
where the first isomorphism is \cref{smooth-induction-module}, the third is \cref{quotients-isomorphic}, and the last is because $\frg$ acts trivially on $V'$. Thus \eqref{pq-sufficient} is equivalent to
\begin{equation} \label{pq-sufficient2}
D(\frg,Q) \otimes_{D(\frg,P)} \left( \tilde{M} \otimes_E V' \right) = \tilde{M} \otimes_E \left( D(\frg,Q) \otimes_{D(\frg,P)} V' \right). 
\end{equation}
Observe that we have natural isomorphisms 
\[ \begin{aligned}
\Hom_{D(\frg,Q)} & \left (D(\frg,Q) \otimes_{D(\frg,P)} \left(\tilde{M} \otimes_E V' \right),- \right) \\
&= \Hom_{D(\frg,P)} \left (\tilde{M} \otimes_E V',- \right) \\
&= \Hom_{D(\frg,P)} \left(V',\Hom_E(\tilde{M},-) \right) \\
&= \Hom_{D(\frg,Q)} \left( D(\frg,Q) \otimes_{D(\frg,P)} V',\Hom_E(\tilde{M},-) \right) \\
&= \Hom_{D(\frg,Q)} \left(\tilde{M} \otimes_E \left(D(\frg,Q) \otimes_{D(\frg,P)} V' \right), - \right)
\end{aligned} \]
where we have used the adjunction of \cref{DgP-adjunction} twice (for the second and final steps). The fully faithfulness of the Yoneda embedding thus implies \eqref{pq-sufficient2}. By unwinding Yoneda's lemma and the above isomorphisms, we find that the isomorphism \eqref{pq-sufficient2} is given explicitly by $1 \otimes m \otimes v \mapsto m \otimes 1 \otimes v$. 
\end{proof}

%% file: examples.tex
\section{Two Examples} \label{examples}

As usual, let $(\bbG, \bbT)$ be a connected split reductive group over $F$ and let $\bbP$ be a parabolic subgroup of $\bbG$ containing $\bbT$. Fix $\log \in \Logs(T)$. As in the introduction, we write $\cF^G_P(M,V)$ for the admissible locally analytic representation $\vF^G_P(M,V)'$. We begin with a general calculation. 

\setcounter{subsection}{1}
\begin{lemma} \label{induction-globalization}
Suppose $M$ is a finite dimensional $\frp$-module and $\alpha : P \to E^\times$ is a smooth character. Then 
\[ \cF^G_P(\Ind_\frp^\frg(M), \alpha) = \Ind_P^G(\Lift(M, \log) \otimes \alpha'). \]
\end{lemma}

\begin{proof}
The $\frp$-linear map $M \to \Ind_\frp^\frg(M) = U(\frg) \otimes_{U(\frp)} M$ induces a morphism of locally analytic representations $\Lift(M, \log) \to \Lift(\Ind_\frp^\frg(M), \log)$, which we regard as a homomorphism of $D(P)$-modules. Tensoring with $\alpha'$ yields another $D(P)$-module homomorphism whose codomain is naturally a $D(\frg,P)$-module (cf. \cref{tensor-DgP-module-construction}), so there is a natural $D(\frg,P)$-linear map 
\begin{equation} \label{crux} 
\begin{tikzcd} D(\frg,P) \otimes_{D(P)} \left( \Lift(M, \log) \otimes \alpha' \right) \ar{r} & \Lift(\Ind_\frp^\frg(M), \log) \otimes \alpha'. \end{tikzcd} \end{equation}
We claim that \eqref{crux} is an isomorphism. To see this, observe that $D(\frg,P) = U(\frg) \otimes_{U(\frp)} D(P)$ as left $\frg$-modules, so the domain is 
\[ D(\frg,P) \otimes_{D(P)} \left( \Lift(M, \log) \otimes \alpha' \right) = U(\frg) \otimes_{U(\frp)} \left( \Lift(M, \log) \otimes \alpha' \right) = U(\frg) \otimes_{U(\frp)} M \]
as $\frg$-modules, since $\Lift(M, \log) = M$ and $\alpha' = E$ as $\frp$-modules. On the other hand, we also have  
\[ \Lift(\Ind_\frp^\frg(M), \log) \otimes \alpha' = \Ind_\frp^\frg(M) \]
as $\frg$-modules, and under these $\frg$-linear identifications, the map \eqref{crux} becomes the identity map. 

To conclude, observe that the isomorphism \eqref{crux} produces the first isomorphism in the following sequence of isomorphisms which prove the desired result. 
\[ \begin{aligned} \vF(\Ind_\frp^\frg(M), \alpha) &= D(G) \otimes_{D(P)} (\Lift(M, \log) \otimes \alpha') \\
&= D(G) \tilde{\otimes}_{D(P)}\, (\Lift(M, \log) \otimes \alpha') \\
&= \Ind_P^G(\Lift(M, \log) \otimes \alpha')'
\end{aligned} \]
Here, we use arguments similar to those in the proof of \cref{smooth-induction-module} (now using finite dimensionality) for the second and third isomorphisms. 
\end{proof}

This brings us to our examples. 

\begin{example}[Breuil's representations] \label{breuil}
Let $F = \bbQ_p$ and let $E$ be a finite extension. Let $\bbG = \mathrm{GL}_2$. Let $\bbP$ be the Borel of upper triangular matrices and $\bbT \subseteq \bbG$ the maximal torus of diagonal matrices. 

Let $M = Em_1 \oplus Em_2$ be the 2-dimensional $\frp$-module where \[ \phi_M \begin{pmatrix} x & * \\ 0 & y \end{pmatrix} : \begin{aligned} m_1 &\mapsto 0 \\ m_2 &\mapsto (x-y)m_1 \end{aligned}. \]
Identifying $\End(M) = \M_2(E)$, we can also write 
\[ \phi_M : \begin{pmatrix} x & * \\ 0 & y \end{pmatrix} \mapsto \begin{pmatrix} 0 & x-y \\ 0 & 0 \end{pmatrix}. \]
Observe that $M$ is the inflation of a $\frt$-module and its only weight is 0. 

If we choose $\mathscr{L} \in E$, there is a unique logarithm map $\log_{\mathscr{L}} : \bbQ_p^\times \to E$ such that $\log_{\mathscr{L}}(p) = \mathscr{L}$. We can use this to construct a logarithm $\log \in \Logs(T)$ by
\[ \log : \begin{pmatrix} a & 0 \\ 0 & d \end{pmatrix} \mapsto \begin{pmatrix} \log_{\mathscr{L}}(a) & 0 \\ 0 & \log_{\mathscr{L}}(d) \end{pmatrix}. \]
Then $\Lift(M, \log)$ is precisely the representation of $P$ that is denoted $\sigma(\mathscr{L})$ in \cite[section 2.1]{BreuilInv}, and \cref{induction-globalization} implies that 
\[ \vF^G_P(\Ind_\frp^\frg(M))' = \Ind_P^G(\sigma(\mathscr{L})). \]
More generally, for an integer $k \geq 2$, let $N(k) = En$ denote the 1-dimensional $\frp$-module given by 
\[ \phi_{N(k)} : \begin{pmatrix} x & * \\ 0 & y \end{pmatrix} \mapsto (k-2)y, \]
and let $\alpha_k : P \to E^\times$ be the smooth character \[ \alpha_k : \begin{pmatrix} a & * \\ 0 & b \end{pmatrix} \mapsto |ab|^{-(k-2)/2}. \]
Note that if $k$ is odd, we must assume that $E$ is large enough that $\sqrt{p} \in E$ in order for this character to be defined. Then $\Lift(M \otimes N(k), \log) \otimes \alpha_k'$ is precisely the representation of $P$ denoted $\sigma(k, \mathscr{L})$ in \cite[section 2.1]{BreuilInv}, and \cref{induction-globalization} implies that 
\[ \cF_P^G(\Ind_{\frp}^{\frg}(M \otimes N(k)), \alpha_k) = \Ind_P^G(\sigma(k, \mathscr{L})). \]
This representation has the representation $\Sigma(k,\sL)$ of \cite[section  2]{BreuilInv} as a subquotient. Note, however, that $\Sigma(k,\sL)$ is not itself in the image of the functor $\cF^G_P$, by (iii) of the following: 
\end{example}

\begin{prop} 
We continue to use the same notation as in \cref{breuil}. Let $\chi$ be a smooth character of $T$.
\begin{enumerate}[(i)]
\item If $L \in \cO^{\frp,\infty}_\alg$ is simple and finite dimensional, then $\cF^G_P(L,\chi)$ is infinite dimensional and locally algebraic. The topological Jordan-H\"older length\footnotemark\, of $\cF^G_P(L,\chi)$ is one or two. In the latter case, it has a finite dimensional and an infinite dimensional topological Jordan-H\"older factor.

\footnotetext{An admissible $W \in \Rep_G^\an$ has \emph{topological Jordan-H\"older length $n$} if there is a filtration 
\[ 0 = W_0 \subsetneq W_1 \subsetneq \cdots \subsetneq W_n = W \]
by closed subrepresentations where each quotient $W_i/W_{i-1}$ has no  nonzero proper closed $G$-invariant subspaces. Then each $W_i/W_{i-1}$ is called a \emph{topological Jordan-H\"older factor} of $W$. 

The topological Jordan-H\"older length and the multi-set of isomorphism classes of topological Jordan-H\"older factors are independent of the filtration. This can be seen by considering the coadmissible $D(G)$-module $M = W'$. The filtration of $W$ gives rise to a filtration 
\[ 0 = M_0 \subsetneq M_1 \subsetneq\cdots \subsetneq M_n = M \]
of $M$ by closed submodules with topologically simple successive quotients. This implies that any submodule of $M$ is closed. To see this, note that it follows from \cite[3.4(iv)]{ST03} that any coadmissible module which is topologically simple (ie, has no nonzero proper closed submodules) is also algebraically simple (ie, has no nonzero proper submodules). Induction on $n$ then implies that any submodule of $M$ is closed. Thus $n$ is equal to the length of $M$ as an abstract $D(G)$-module, and the multi-set of (isomorphism classes of) the quotients $M_i/M_{i-1}$ are the abstract Jordan-H\"older factors of $M$, so both are independent of the filtration.}

\item  If $L \in \cO^\infty_\alg$ is simple and infinite dimensional, then $\cF^G_P(L,\chi)$ is topologically irreducible and not locally algebraic.

\item $\Sigma(k,\sL)$ is not isomorphic as a locally analytic representation to a representation of the form $\cF^G_P(M,V)$ where $M$ ia in $\cO^\infty_\alg$ and $V$ is a smooth strongly admissible representation of $T$.
\end{enumerate}
\end{prop}

\begin{proof}
(i) If $L$ is finite dimensional, it is contained in $\cO^\frg_\alg$ and hence lifts to an algebraic representation of $G$. By \cref{PQthm}, we see that \[ \cF^G_P(L,\chi) = \cF^G_G(L,i(\chi)) = L' \otimes \ind^G_P(\chi) \] is locally algebraic. The smooth representation $\ind^G_P(\chi)$ is well-known to have length one or two. In the latter case, it has a one dimensional and an infinite dimensional Jordan-H\"older factor. If $W$ is an irreducible smooth representation of $G$, then $L' \otimes W$ is irreducible by \cite[4.2.8\footnotemark]{EmertonA}. \footnotetext{To apply this result, note that $L$, being $E$-split, is absolutely irreducible, so $\End_{\bbG}(L') = E$.} 

\mbox{}

\noindent (ii) If $L$ is an infinite dimensional simple module, it is an irreducible Verma module. Hence $\frp$ is maximal for $L$ in the sense of \cite[5.2]{OrlikStrauchJH}. It follows from \cite[5.8]{OrlikStrauchJH} that $\cF^G_P(L,\chi)$ is topologically irreducible.

\mbox{}

\noindent (iii) Let us assume, to the contrary, that there is a topological isomorphism $\Sigma(k,\sL) \cong \cF^G_P(M,V)$ of $G$-representations. The representation $\Sigma(k,\sL)$ has three topological Jordan-H\"older factors, $W_1$, $W_2$ and $W_3$, where $W_1$ is finite dimensional, $W_2$ is locally algebraic and infinite dimensional, and $W_3$ is a topologically irreducible locally analytic principal series representation which is not locally algebraic \cite[2.1.2]{BreuilInv}. Moreover, by \cite[2.4.2]{BreuilInv}, $W_1$ is the unique topologically irreducible quotient (by a closed subrepresentation) of $\Sigma(k,\sL)$, and $W_2$ is the unique closed topologically irreducible subrepresentation of $\Sigma(k,\sL)$. It follows from the explicit description of these Jordan-H\"older factors in \cite[2.1.2]{BreuilInv} that they stay irreducible after passing from $E$ to a finite extension.

Let us say that a topological Jordan-H\"older factor is of {\it type 1} (resp. {\it 2}, {\it 3}) if it is finite dimensional (resp. infinite dimensional and locally algebraic, infinite dimensional and not locally algebraic).

Let $L$ be a simple subquotient of $M$. Exactness of $\cF^G_P$ implies that the topological Jordan-H\"older factors of $\cF^G_P(L,V)$ are also topological Jordan-H\"older factors of $\cF^G_P(M,V)$.  

\mbox{}

\noindent {\it Claim.} $V$ is one dimensional (ie, a character). 

\mbox{}

\noindent {\it Proof of the claim.}
Suppose that is not the case. After extending the coefficient field $E$, we find $E$-valued smooth characters $\chi_1$ and $\chi_2$ of $T$ such that $\chi_1 \hra V$ and $\chi_2 \hra V/\chi_1$. If $L$ is finite dimensional, then, by part (i), both $\cF^G_P(L,\chi_1)$ and $\cF^G_P(L,\chi_2)$ contribute topological Jordan-H\"older factors of type 2 to $\cF^G_P(L,V)$, hence to $\cF^G_P(M,V)$, but there is only one such.  Similarly, if $L$ is infinite dimensional, then, by part (ii), both $\cF^G_P(L,\chi_1)$ and $\cF^G_P(L,\chi_2)$ contribute topological Jordan-H\"older factors of type 3 to $\cF^G_P(L,V)$, hence to $\cF^G_P(M,V)$, but there is only one such.\qed

\mbox{}

We henceforth write $V = \chi$. It follows from parts (i) and (ii) that $M$ must have exactly one finite dimensional Jordan-H\"older factor, and exactly one infinite dimensional Jordan-H\"older factor which we call $L$. Thus $M$ has length two and $L$ must be either a submodule or a quotient of $M$. An embedding $L \hra M$ gives rise to a continuous surjection $\cF^G_P(M,\chi) \twoheadrightarrow \cF^G_P(L,\chi)$. 
By (ii), the target of this map is irreducible and of type 3, but $W_1$ is the only topologically irreducible quotient of $\Sigma(k,\sL)$ and it is of type 1. On the other hand, a surjection $M \twoheadrightarrow L$ gives rise to closed embedding $\cF^G_P(L,\chi) \hra \cF^G_P(M,\chi)$. By (ii), the domain of this map is irreducible and of type 3, but $W_2$ is the unique closed topologically irreducible subrepresentation of $\Sigma(k,\sL)$ and it is of type 2. 
\end{proof}

\begin{remark} \label{choosing-logarithm}
Let us continue to use the same notation as in \cref{breuil}. Here's another perspective on the choice of logarithms that occurred above. We can choose coordinates on $\bbT$ using the ``standard'' isomorphism $\bbT \to \bbG_m^2$ given by $\diag(a,b) \mapsto (a,b)$, but we can also use other isomorphisms to choose coordinates. Suppose we choose coordinates using the isomorphism $\zeta : \bbT \to \bbG_m^2$ given by $\diag(a,b) \mapsto (ab^{-1}, b)$. If we choose $\sL_1, \sL_2 \in E$, we get a logarithm $\log : (\bbQ_p^\times)^2 \to E^2$ on $\bbG_m^2(\Q_p) = (\Q_p^\times)^2$ given by $\log(a,b) = (\log_{\sL_1}(a), \log_{\sL_2}(b))$. Now consider the following commutative diagram. 
\[ \begin{tikzcd} \frt_E \ar{r}{d\zeta} & E^2 \\ \frt \ar{r}{d\zeta} \ar[dashed]{d}[swap]{\exp} \ar{u} & \Q_p^2 \ar{u} \ar[dashed]{d}[swap]{\exp} \\ T \ar{r}[swap]{\zeta} & (\Q_p^\times)^2 \ar[bend right=60]{uu}[swap]{\log} \end{tikzcd} \] 
We can use this to define a logarithm $\log : T \to \frt_E$ on $T$ given by \[ \log = (d\zeta)^{-1} \circ \log \circ \, \zeta, \] and this is in fact a logarithm since
\[ \log \circ \exp = (d\zeta^{-1} \circ \log \circ\, \zeta) \circ \exp = d\zeta^{-1} \circ \log \circ \exp \circ\, d\zeta = d\zeta^{-1} \circ \id \circ d\zeta = \id, \]
where we abusively write $\id$ to denote the natural maps $\frt \to \frt_E$ and $\Q_p^2 \to E^2$. Explicitly, this logarithm is given by 
\[ \begin{aligned} \log \begin{pmatrix} a & 0 \\ 0 & b \end{pmatrix} &= \begin{pmatrix} \log_{\sL_1}(ab^{-1}) + \log_{\sL_2}(b) & 0 \\ 0 & \log_{\sL_2}(b) \end{pmatrix} \\
&= \begin{pmatrix} \log_{\sL_1}(a) + \gamma(a,b) & 0 \\ 0 & \log_{\sL_2}(b) \end{pmatrix} \end{aligned} \]
where $\gamma$ is the smooth function \[ \gamma(a,b) = \log_{\sL_2}(b) - \log_{\sL_1}(b). \]
If we fix $\sL_1 = \sL$, the above construction of a logarithm on $T$ applied with any choice of $\sL_2 \in E$ will furnish us with a logarithm $\log \in \Logs(T)$ with the property that $\Lift(M \otimes N(k), \log) \otimes \alpha_k'$ is Breuil's representation $\sigma(k, \sL)$. Also, in preparation for \cref{schraen} below, we remark that it is not necessary for $\zeta$ to be an isomorphism; it is sufficient for it to be an \'etale homomorphism of group schemes (since this will imply that $d\zeta$ is invertible).
\end{remark}

\begin{example}[Schraen's representations] \label{schraen}
Suppose $F = \bbQ_p$ and $E$ is a finite extension. Let $\bbG = \mathrm{GL}_3$. Let $\bbP$ be the Borel of upper triangular matrices and $\bbT$ the maximal torus of diagonal matrices. 

Let $M = Em_1 \oplus Em_2 \oplus Em_3$ be the 3-dimensional representation of $\frp$ given by 
\[ \phi_M : \begin{pmatrix} x & * & * \\ 0 & y & * \\ 0 & 0  & z \end{pmatrix} \mapsto \begin{pmatrix} 0 & -2x+y+z & -x-y+2z \\ 0 & 0 & 0 \\ 0 & 0 & 0 \end{pmatrix}, \]
where we have identified $\End(M) = \M_3(E)$ using the basis $m_1, m_2, m_3$. Again, $M$ is the inflation of a representation of $\frt$ and its only weight is 0. 
We can use a finite \'etale homomorphism $\zeta : \bbT \to \bbG_{m}^3$ of the form 
\[ \begin{pmatrix} a & 0 & 0\\ 0 & b &  0 \\ 0 & 0 & c \end{pmatrix} \mapsto (a^{-2}bc,a^{-1}b^{-1}c^2,*) \]
to choose a logarithm on $T$ (eg, take $* = c$). More precisely, given $\mathscr{L}, \mathscr{L}' \in E$, choose any logarithm on $\bbG_m^3(\Q_p) = (\Q_p^\times)^3$ which is given by $\log_{\mathscr{L}'}$ on the first coordinate and $\log_{\mathscr{L}}$ on the second (there are infinitely many such logarithms). Via $\zeta$, this induces a logarithm $\log \in \Logs(T)$ as in \cref{choosing-logarithm}. This logarithm has the property that $\Lift(M, \log)$ is precisely the representation of $P$ that is denoted $\sigma(\mathscr{L}, \mathscr{L}')$ in \cite[remarque 5.14]{SchraenGL3}. For example, we can take the logarithm $\log : T \to \frt_E$ given by 
\[ \log \begin{pmatrix} a & 0 & 0\\ 0 & b &  0 \\ 0 & 0 & c \end{pmatrix} =  \begin{pmatrix} \log_{\sL'}(a) + \gamma(a,b,c) & 0 & 0\\ 0 & \log_{\sL'}(b) + 2\gamma(a,b,c) &  0 \\ 0 & 0 & \log_{\sL'}(c) \end{pmatrix}\]
where $\gamma$ is the smooth function
\[ \gamma(a,b,c) = -\frac{1}{3}(\log_{\sL}(a^{-1}b^{-1}c^2)-\log_{\sL'}(a^{-1}b^{-1}c^2)). \]
Thus \cref{induction-globalization} implies that 
\[ \cF^G_P(\Ind_\frp^\frg(M)) = \Ind_P^G(\sigma(\sL,\sL')). \]
For $\lambda$ a dominant weight, the representation $\Sigma(\lambda,\sL,\sL')$ of \cite[5.12--14]{SchraenGL3} is thus a subquotient of a representation in the image of $\cF^G_P$, but \cref{breuil} leads one to expect that $\Sigma(\lambda, \sL, \sL')$ is not itself in the image of $\cF^G_P$. 
\end{example}

%% file: adjunction.tex
\section{Tensor-hom adjunction for \texorpdfstring{$D(H)$}{D(H)}-modules}

Let $H$ be a locally analytic group. In the following, whenever we consider separately or jointly continuous $D(H)$-modules, we assume that the underlying topological vector space is locally convex. If $M$ is a locally analytic representation, there exists a separately continuous $D(H)$-module structure on $M$ which has the property that $\delta_h \bdot m = h \bdot m$ for all $h \in H$ and $m \in M$ \cite[proposition 3.2]{ST02b}. 

We will assume in this section that $M$ is locally finite dimensional (ie, that it is the colimit of its finite dimensional subrepresentations). Our goal is to show that, if $X$ is any $D(H)$-module, then we can endow $M \otimes X$ and $\Hom(M, X)$ with natural $D(H)$-module structures in such a way that $M \otimes -$ is left adjoint to $\Hom(M, -)$. This would be clear if $D(H)$ were a Hopf algebra, but it is not quite a Hopf algebra \cite[appendix to section 3]{ST05}. 

Note that everywhere in this section, we do \emph{not} ask the $D(H)$-module $X$ to come equipped with a topology; even if it is just an abstract $D(H)$-module, we obtain natural $D(H)$-module structures on $M \otimes X$ and $\Hom(M, X)$ that fit into a tensor-hom adjunction. 

\subsection{Some functional analysis}

\begin{lemma} \label{tensor-fd}
Suppose $V$ is a locally convex vector space and $W$ is a \emph{finite dimensional} vector space. Then the inductive and projective topologies on $V \otimes W$ coincide. In other words, if $U$ is a locally convex vector space and $\sigma : V \times W \to U$ is bilinear and separately continuous, then it is automatically jointly continuous. 
\end{lemma}

\begin{proof}
Let $N$ be an open lattice in $U$. Choose a basis $w_1, \dotsc, w_n$ of $W$. Then $\sigma(-, w_i) : V \to U$ is continuous for all $i$, so there exist open lattices $L_i \subseteq V$ such that $\sigma(L_i, w_i) \subseteq N$. Then let $L = \cap L_i$, so that $\sigma(L, w_i) \subseteq N$. Let $M$ be the $\mathscr{O}_E$-submodule of $W$ generated by $w_1, \dotsc, w_n$. Then $M$ is an open lattice in $W$. If $v \in L$ and if $a_1, \dotsc, a_n \in \mathscr{O}_E$, then 
\[ \sigma\left(v, \sum a_i w_i \right) = \sum_i a_i \sigma(v, w_i). \]
We know that $\sigma(v, w_i) \in N$, and $N$ is an $\mathscr{O}_E$-submodule of $U$, so $\sigma(v, \sum a_i w_i) \in N$ as well. Thus $\sigma(L, M) \subseteq N$, proving that $\sigma$ is continuous. 
\end{proof}

\begin{remark}
We frequently use \cref{tensor-fd} tacitly. Whenever $V$ is a locally convex vector space and $W$ is finite dimensional, we regard $V \otimes W$ as a locally convex vector space with the ``inductive = projective'' topology without specifying a subscript $i$ or $\pi$. 
\end{remark}

\begin{remark} \label{lfd-jointly-continuous}
If $M$ is a finite dimensional locally analytic representation of $H$, the separately continuous $D(H)$-module structure $D(H) \times M \to M$ is automatically continuous by \cref{tensor-fd}. Taking colimits, we find the same is true when $M$ is only locally finite dimensional. 
\end{remark}

\subsection{Constructions and functoriality}

We begin by constructing a module structure over $D(H)$ on $M \otimes X$, where $M$ is a locally finite dimensional locally analytic representation of $H$ and $X$ is a $D(H)$-module. \Cref{construction-fd} does this when $M$ is finite dimensional, and \cref{construction} upgrades this to the general case. Intermediate to this we have \cref{functoriality-M-fd} which proves that the construction of the $D(H)$-module structure on $M \otimes X$ is functorial in $M$ when restricted to the category of finite dimensional locally analytic repersentations of $H$ (and again, \cref{functoriality-M} upgrades this statement to the general case). 

\begin{lemma} \label{construction-fd}
Suppose $M$ is a finite dimensional locally analytic representation of $H$. For any $D(H)$-module $X$, there exists a natural $D(H)$-module structure on the tensor product $M \otimes X$ with the property that 
\[ \delta_h \bdot (m \otimes x) = (h \bdot m) \otimes (\delta_h \bdot x) \]
for all $h \in H$, $m \in M$, and $x \in X$. Moreover, if $X$ is a separately (resp. jointly) continuous $D(H)$-module, then so is $M \otimes X$. 
\end{lemma}

\begin{proof}
The tensor product $M \otimes X$ naturally has the structure of a module over $D(H) \otimes D(H)$, given by 
\[ (\mu_1 \otimes \mu_2)(m \otimes x) = (\mu_1 \bdot m) \otimes (\mu_2 \bdot x). \]
Observe that the structure map $D(H) \otimes D(H) \to \End(M \otimes X)$ factors as 
\[ \begin{tikzcd} D(H) \otimes D(H) \ar{r} & \End(M) \otimes D(H) \ar{r} & \End(M \otimes X), \end{tikzcd} \]
where the first map is the structure map $D(H) \to \End(M)$ tensored with $D(H)$, and the second map is given by \[ \sigma \otimes \mu \mapsto [m \otimes x \mapsto \sigma(m) \otimes \mu \bdot x]. \]
Since $M$ is a separately continuous $D(H)$-module, the map $D(H) \to \End(M)$ is continuous, so the map $D(H) \otimes D(H) \to \End(M) \otimes D(H)$ is also continuous. Since $\End(M)$ is finite dimensional and $D(H)$ is complete, the tensor product $\End(M) \otimes D(H)$ is also complete. Thus $D(H) \otimes D(H) \to \End(M) \otimes D(H)$ factors naturally through $D(H) \, \hat{\otimes}_i \, D(H)$. 
\[ \begin{tikzcd} D(H) \,\hat{\otimes}_i\, D(H) \ar{r} & \End(M) \otimes D(H) \ar{r} & \End(M \otimes X). \end{tikzcd} \]
Now note that $D(H \times H) = D(H) \, \hat{\otimes}_i \, D(H)$ \cite[proposition A.3]{ST05}, and that the diagonal map $\Delta : H \to H \times H$ induces a continuous homomorphism $\Delta_* : D(H) \to D(H \times H)$ of algberas. The composite
\[ \begin{tikzcd} D(H) \ar{r}{\Delta_*} & D(H \times H) = D(H) \, \hat{\otimes}_i\, D(H) \ar[out=0,in=180, to path={..controls +(5,-0.4) and +(-5,0.4).. (\tikztotarget)}]{d} \\
& \End(M) \otimes D(H) \ar{r} & \End(M \otimes X) \end{tikzcd} \]
then defines the structure of a $D(H)$-module on $M \otimes X$. It is easy to see that this $D(H)$-module structure is given by the stated formula on delta distributions.\footnote{To know how a distribution in $D(H)$ act on $M \otimes X$, one would need to know how to transfer the image of that distribution under $\Delta_*$ through the isomorphism $D(H \times H) = D(H) \, \hat{\otimes}_i \, D(H)$. This is clear for delta distributions. For distributions in the image of $U(\frh) \to D(H)$, see \cref{hopf-algebra-map} below. For other distributions, this is not entirely clear.}

If $X$ happens to be a separately continuous $D(H)$-module, we would like to show that the resulting $D(H)$-module structure on $M \otimes X$ is also separately continuous. In other words, we need to show that the maps $D(H) \to \End(M \otimes X)$ and $M \otimes X \to \Hom(D(H), M \otimes X)$ land in the spaces $\End^\cts(M \otimes X)$ and $\Hom^\cts(D(H), M \otimes X)$ of continuous linear maps, respectively. 

To show that $D(H) \to \End(M \otimes X)$ lands inside $\End^\cts(M \otimes X)$, we note that, by using the above factorization, it is sufficient to show that $\End(M) \otimes D(H) \to \End(M \otimes X)$ lands inside $\End^\cts(M \otimes X)$. In other words, we have to show that, for fixed $\sigma \in \End(M)$ and $\mu \in D(H)$, the map $M \otimes X \to M \otimes X$ given by $m \otimes x \mapsto \sigma(m) \otimes \mu \bdot x$ is continuous. But this map factors as 
\[ \begin{tikzcd} M \otimes X \ar{r}{\sigma \otimes X} & M \otimes X \ar{r}{M \otimes \mu} & M \otimes X \end{tikzcd} \]
and each of these maps is continuous (the former since $M$ is finite dimensional, the latter since $X$ is separately continuous). Thus the composite is continuous as well. 

Next, we need to show that the image of $M \otimes X \to \Hom(D(H), M \otimes X)$ is inside $\Hom^\cts(D(H), M \otimes X)$. Fix $m \otimes x$. We want to show that the map $D(H) \to M \otimes X$ given by $\mu \mapsto \mu \bdot (m \otimes x)$ is continuous. Observe that this map factors as follows. 
\[ \begin{tikzcd} D(H) \ar{r}{\Delta_*} & D(H \times H) = D(H) \, \hat{\otimes}_i \, D(H) \ar{r}{m \otimes D(H)} & M \otimes D(H) \ar{r}{M \otimes x} & M \otimes X \end{tikzcd} \]
The first two maps are continuous, and the last map is continuous by separate continuity of $X$. To see that $D(H) \, \hat{\otimes}_i \, D(H) \to M \otimes D(H)$ is continuous, note that the map $D(H) \to M$ given by $\mu \mapsto \mu \bdot x$ is continuous by separate continuity of $M$, so $D(H) \otimes_i D(H) \to M \otimes D(H)$ is continuous as well. Since $M$ is finite dimensional, we know that $M \otimes D(H)$ is complete, so this continuous map factors continuously through $D(H) \, \hat{\otimes}_i \, D(H)$, as desired. 

Now suppose $X$ is \emph{jointly} continuous. Then $D(H) \otimes_{\pi} X \to X$ is continuous. Since $M$ is finite dimensional, separate continuity of $D(H) \times M \to M$ implies continuity, so $D(H) \otimes_\pi M \to M$ is also continuous. Thus the map 
\[ (D(H) \otimes_\pi D(H)) \otimes_\pi (M \otimes_\pi X) = (D(H) \otimes_\pi M) \otimes_\pi (D(H) \otimes_\pi X) \to M \otimes_\pi X \]
is continuous. Since the inductive topology is finer than the projective topology, this implies that 
\[ \begin{tikzcd} (D(H) \otimes_i D(H)) \otimes_\pi (M \otimes X) \ar{r} & M \otimes X \end{tikzcd} \]
is continuous. We noted above that the action of $D(H) \otimes D(H)$ on $M \otimes X$ factors through $D(H) \, \hat{\otimes}_i \, D(H)$, so the above map factors through a continuous map
\[ \begin{tikzcd} (D(H) \,\hat{\otimes}_i \, D(H)) \otimes_\pi (M \otimes X) \ar{r} & M \otimes X. \end{tikzcd} \]
Now note that $D(H) \to D(H) \,\hat{\otimes}_i \, D(H)$ is continuous, so the first map in the composite
\[ \begin{tikzcd} D(H) \otimes_\pi (M \otimes X) \ar{r} & (D(H) \,\hat{\otimes}_i \, D(H)) \otimes_\pi (M \otimes X) \ar{r} & M \otimes X. \end{tikzcd} \]
is also continuous. This proves that the $D(H)$-module structure on $M \otimes X$ is jointly continuous. 
\end{proof}

\begin{lemma} \label{functoriality-M-fd}
Suppose $f : M \to N$ is a homomorphism of finite dimensional locally analytic representations of $H$. If $X$ is a $D(H)$-module, then $f \otimes X : M \otimes X \to N \otimes X$ is also $D(H)$-linear. 
\end{lemma}

\begin{proof}
We want to show that the following diagram commutes. 
\[ \begin{tikzcd} D(H) \ar{r} \ar{d} & \End(M \otimes X) \ar{d}{f_*} \\ \End(N \otimes X) \ar{r}[swap]{f^*} & \Hom(M \otimes X, N \otimes X)  \end{tikzcd} \]
In this proof, we use $f_*$ to denote maps induced by $f$ under covariant functoriality of $\Hom$, and $f^*$ to denote maps induced by $f$ under contravariant functoriality of $\Hom$. For example, in the above diagram, $f_*$ denotes postcomposition with $f \otimes X : M \otimes X \to N \otimes X$, and $f^*$ denotes precomposition with $f \otimes X$. 

Note that $D(H) \to \End(M \otimes X)$ factors through $\End(M) \otimes D(H)$ and $D(H) \to \End(N \otimes X)$ factors through $\End(N) \otimes X$, and the black part of the following diagram clearly commutes. 
\[ \begin{tikzcd}
\color{lightgray} D(H) \ar[color=lightgray]{rr} \ar[color=lightgray]{d} & & \End(M) \otimes D(H) \ar{d} \ar{dl}{f_*} \\
\End(N) \otimes D(H) \ar{d} \ar{r}{f^*} & \Hom(M, N) \otimes D(H) \ar{dr} & \End(M \otimes X) \ar{d}{f_*} \\
\End(N \otimes X) \ar{rr}[swap]{f^*} & & \Hom(M \otimes X, N \otimes X)  \end{tikzcd} \]
Thus it is sufficient to show that the following diagram commutes. 
\[ \begin{tikzcd} D(H) \ar{r} \ar{d} & \End(M) \otimes D(H) \ar{d}{f_*} \\ \End(N) \otimes D(H) \ar{r}[swap]{f^*} & \Hom(M, N) \otimes D(H)  \end{tikzcd} \]
These are all continuous maps. Thus it sufficies to check commutativity after restricting to $E[H]$, where this is clear. 
\end{proof}

\begin{lemma} \label{construction}
Suppose $M$ is a locally finite dimensional locally analytic representation of $H$. For any $D(H)$-module $X$, there exists a natural $D(H)$-module structure on the tensor product $M \otimes X$ with the property that 
\[ \delta_h \bdot (m \otimes x) = (h \bdot m) \otimes (\delta_h \bdot x) \]
for all $h \in H$, $m \in M$, and $x \in X$. Moreover, if $X$ is a separately (resp. jointly) continuous $D(H)$-module, then so is $M \otimes X$. 
\end{lemma}

\begin{proof}
We know that $M$ is the colimit\footnote{This colimit, and all others that occur in this section, are really just directed unions.} of its finite dimensional subrepresentations $N$. Note that 
\[ M \otimes X = (\colim N) \otimes X = \colim (N \otimes X). \]
Each $N \otimes X$ has a $D(H)$-module structure as in \cref{construction-fd}, and the transition maps are all $D(H)$-linear by \cref{functoriality-M-fd}. Thus $M \otimes X$ acquires a $D(H)$-module structure from each $N \otimes X$. It too is given by the same formula on delta distributions. Moreover, if $X$ is separately (resp. jointly) continuous, then $M \otimes X$ inherits the same continuity property from each $N \otimes X$. 
\end{proof}

\begin{lemma} \label{functoriality-M}
Suppose $f : M \to N$ is a homomorphism of locally finite dimensional locally analytic representations of $H$. If $X$ is a $D(H)$-module, then $f \otimes X : M \otimes X \to N \otimes X$ is also $D(H)$-linear. 
\end{lemma}

\begin{proof}
Suppose $M' \subseteq M$ and $N' \subseteq N$ are finite dimensional subrepresentations such that $f(M') \subseteq N'$. By \cref{functoriality-M-fd}, we know that $M' \otimes X \to N' \otimes X$ is $D(H)$-linear. As $M'$ and $N'$ vary, the transition maps are also $D(H)$-linear by \cref{functoriality-M-fd}. Taking a colimit over all such $M'$ and $N'$ thus yields the result.  
\end{proof}

\begin{lemma} \label{hopf-algebra-map}
The following diagram commutes. 
\[ \begin{tikzcd} U(\frh) \ar{d} \ar{rr} & & U(\frh) \otimes U(\frh) \ar{d} \\ 
D(H) \ar{r} & D(H \times H) \ar[equals]{r} & D(H) \, \hat{\otimes}_i \, D(H) \end{tikzcd} \]
Here, $U(\frh) \to U(\frh) \otimes U(\frh)$ is the comultiplication on $U(\frh)$, $D(H) \to D(H \times H)$ is the continuous algebra  homomorphism induced by the diagonal map $H \to H \times H$, and $D(H \times H) = D(H) \, \hat{\otimes}_i \, D(H)$ is the isomorphism of \cite[proposition A.3]{ST05}. 
\end{lemma}

\begin{proof}
The comultiplication on $U(\frh)$ factors as $U(\frh) \to U(\frh \oplus \frh)$ induced by the diagonal map $\frh \to \frh \oplus \frh$ and a natural isomorphism $U(\frh \oplus \frh) = U(\frh) \otimes U(\frh)$. 
\[ \begin{tikzcd} U(\frh) \ar{d} \ar{r} & U(\frh \oplus \frh) \ar[equals]{r} \ar{d} & U(\frh) \otimes U(\frh) \ar{d} \\ 
D(H) \ar{r} & D(H \times H) \ar[equals]{r} & D(H) \, \hat{\otimes}_i \, D(H) \end{tikzcd} \]
Since the inclusion $U(\Lie(-)) \to D(-)$ is functorial in its argument (a locally analytic group), we can apply this functoriality to the diagonal map $H \to H \times H$ and conclude that the above diagram commutes.
\end{proof}

\begin{corollary} \label{restrict-tensor-product}
Suppose $M$ is a locally finite dimensional locally analytic representation of $H$. The $D(H)$-module structure on $M \otimes X$  from \ref{construction} restricts to the usual $U(\frh)$-module structure on $M \otimes X$. 
\end{corollary}

\begin{lemma}  \label{associativity}
Suppose $M$ and $N$ are both finite dimensional locally analytic representations of $H$. For any $D(H)$-module $X$, the associativity isomorphism $M \otimes (N \otimes X) = (M \otimes N) \otimes X$ is $D(H)$-linear. 
\end{lemma}

\begin{proof}
Let $\alpha : M \otimes (N \otimes X) \to (M \otimes N) \otimes X$ be the associativity isomorphism \[ m \otimes (n \otimes x) \mapsto (m \otimes n) \otimes x. \] We want to show that the following diagram commutes. 
\[ \begin{tikzcd} D(H) \ar{r} \ar{d} & \End(M \otimes (N \otimes X)) \ar{d}{\alpha_*} \\ \End((M \otimes N) \otimes X) \ar{r}[swap]{\alpha^*} & \Hom(M \otimes (N \otimes X), (M \otimes N) \otimes X)  \end{tikzcd} \]
Note that the structure map $D(H) \to \End(M \otimes (N \otimes X))$ factors through $\End(M) \otimes (\End(N) \otimes D(H))$, while the structure map $D(H) \to \End((M \otimes N) \otimes X)$ factors through $\End(M \otimes N) \otimes D(H)$. It is clear that the black part of the follow diagram commutes. 
\[ \begin{tikzcd} \color{lightgray} D(H) \ar[color=lightgray]{r} \ar[color=lightgray]{d} & \End(M) \otimes (\End(N) \otimes D(H)) \ar{dl} \ar{d} \\
\End(M \otimes N) \otimes D(H) \ar{d} & \End(M \otimes (N \otimes X)) \ar{d}{\alpha_*} \\ 
\End((M \otimes N) \otimes X) \ar{r}[swap]{\alpha^*} & \Hom((M \otimes N) \otimes X, M \otimes (N \otimes X)) \end{tikzcd} \]
Thus it is sufficient to show that the following diagram commutes. 
\[ \begin{tikzcd} D(H) \ar{d} \ar{r} & \End(M) \otimes (\End(N) \otimes D(H)) \ar{dl} \\ \End(M \otimes N) \otimes D(H) \end{tikzcd} \]
These are all continuous maps. Thus it sufficies to check commutativity after restricting to $E[H]$, where this is clear. 
\end{proof}

\begin{lemma} \label{functoriality-X}
Suppose $M$ is a locally finite dimensional locally analytic representation of $H$. If $f : X \to Y$ is a homomorphism of $D(H)$-modules, then $M \otimes f : M \otimes X \to M \otimes Y$ is also $D(H)$-linear. 
\end{lemma}

\begin{proof} 
It is sufficient to consider the case when $M$ is finite dimensional. We want to show that the following diagram commutes. 
\[ \begin{tikzcd} D(H) \ar{r} \ar{d} & \End(M \otimes X) \ar{d}{f_*} \\ \End(M \otimes Y) \ar{r}[swap]{f^*} & \Hom(M \otimes X, M \otimes Y)  \end{tikzcd} \]
By construction, the two maps coming out of $D(H)$ factor through $\End(M) \otimes D(H)$ in the same way, so it is sufficient to show that the following diagram commutes. 
\[ \begin{tikzcd} \End(M) \otimes D(H) \ar{r} \ar{d} & \End(M \otimes X) \ar{d}{f_*} \\ \End(M \otimes Y) \ar{r}[swap]{f^*} & \Hom(M \otimes X, M \otimes Y)  \end{tikzcd} \]
This follows from $D(H)$-linearity of $f$.
\end{proof}

\begin{remark} \label{hom-construction}
Suppose $M$ is a finite dimensional locally analytic representation of $H$ and $X$ is a $D(H)$-module. Then the contragredient $M'$ is also a finite dimensional locally analytic representation of $H$. If we transfer the $D(H)$-module structure on $M' \otimes X$ through the natural isomorphism $M' \otimes X = \Hom(M, X)$, it is given by 
\[ (\delta_h \bdot \sigma)(m) = \delta_h \bdot \sigma( h^{-1} \bdot m) \]
for $h \in H, m \in M$, and $\sigma \in \Hom(M, X)$. To see this, observe that the natural isomorphism $\psi : M' \otimes X \to \Hom(M, X)$ is given by 
\[ \psi(\lambda \otimes x)(m) = \lambda(m)x. \]
If $\psi(\lambda \otimes x) = \sigma$, then 
\[ \begin{aligned} (\delta_h \bdot \sigma)(m) &= \psi(\delta_h \bdot \psi^{-1}(\sigma))(m) \\
&= \psi(\delta_h \bdot (\lambda \otimes x))(m) \\
&= \psi((h \bdot \lambda) \otimes (\delta_h \bdot x))(m) \\
&= (h \bdot \lambda)(m)(\delta_h \bdot x) \\
&= \delta_h \bdot (h \bdot \lambda)(m)x \\
&= \delta_h \bdot \lambda(h^{-1} \bdot m)x \\
&= \delta_h \bdot \sigma(h^{-1} \bdot m).
\end{aligned} \]
If $X$ is a separately (resp. jointly) $D(H)$-module, then so is $\Hom(M, X)$, simply because $M' \otimes X = \Hom(M, X)$ is a topological isomorphism. If $M$ is only \emph{locally} finite dimensional, it is the colimit of its finite dimensional subrepresentations $N$, so 
\[ \Hom(M, X) = \Hom(\colim N, X) = \lim \Hom(N, X) \]
has a natural structure of a $D(H)$-module, induced by the $D(H)$-module structure described above on each $\Hom(N, X)$. And again, if $X$ is a separately (resp. jointly) continuous $D(H)$-module, then $\Hom(M, X)$ inherits the same continuity property from all of the $\Hom(N, X)$. 
\end{remark}

\subsection{Adjunction}

\begin{theorem} \label{adjunction}
Suppose $M$ is a locally finite dimensional locally analytic representation of $H$. Then there is an adjunction of $D(H)$-module-valued functors $(M \otimes -) \dashv \Hom(M, -)$. 
\end{theorem}

\begin{proof}
Suppose first that $M$ is finite dimensional. We know that there is an adjunction $(M \otimes -) \dashv \Hom(M, -)$ of vector-space-valued functors, and we want to upgrade this to an adjunction of $D(H)$-module valued functors. Suppose $X$ is a $D(H)$-module. Since adjunctions are determined by their counit and unit maps, it is sufficient to show that the counit $\epsilon : M \otimes \Hom(M, X) \to X$ and the unit $\eta : X \to \Hom(M, M \otimes X)$ of the vector-space-valued adjunction are both $D(H)$-linear.\footnotemark

\footnotetext{Suppose $X$ and $Y$ are $D(H)$-modules and $f : M \otimes X \to Y$ is $D(H)$-linear. Under the isomorphism $\Hom(M \otimes X, Y) = \Hom(X, \Hom(M, Y))$, the map $f$ on the left corresponds on the right to the composite $f_* \circ \eta$, where $f_* : \Hom(M, M \otimes X) \to \Hom(M, Y)$. Then $D(H)$-linearity of $\eta$ and $f$ implies $D(H)$-linearity of $f_* \circ \eta$, proving that the isomorphism $\Hom(M \otimes X, Y) = \Hom(X, \Hom(M, Y))$ restricts to a well-defined map $\Hom_{D(H)}(M \otimes X, Y) \to \Hom_{D(H)}(X, \Hom(M, Y))$. We use the counit $\epsilon$ to go the other way. }

First let us show that $\epsilon$ is $D(H)$-linear. 
\[ M \otimes \Hom(M, X) = M \otimes (M' \otimes X) = (M \otimes M') \otimes X \]
as $D(H)$-modules by \cref{hom-construction,associativity}. Let $\ev : M \otimes M' \to E$ be the evaluation map. Under the isomorphism $M \otimes \Hom(M,X) = (M \otimes M') \otimes X$, the counit $\epsilon$ corresponds to the map $\ev \otimes X : (M \otimes M') \otimes X \to X$. Since $\ev$ is a homomorphism of representations, it follows from \cref{functoriality-M} that $\ev \otimes X$ is $D(H)$-linear. Thus $\epsilon$ is also $D(H)$-linear. 

Next we show that $\eta$ is $D(H)$-linear in the same way. Note that 
\[ \Hom(M, M \otimes X) = M' \otimes (M \otimes X) = (M' \otimes M) \otimes X \]
as $D(H)$-modules again by \cref{hom-construction,associativity}. Let $\iota : E \to M \otimes M'$ be the coevaluation. Under the isomorphism $\Hom(M, M \otimes X) = (M' \otimes M) \otimes X$, the unit $\eta$ corresponds to the map $\iota \otimes X : X \to (M' \otimes M) \otimes X$. Since $\iota$ is a homomorphism of representations, it follows from \cref{functoriality-M} that $\iota \otimes X$ is $D(H)$-linear. Thus $\eta$ is also $D(H)$-linear. 

This completes the proof when $M$ is finite dimensional. If $M$ is \emph{locally} finite dimensional, it is the colimit of its finite dimensional subrepresentations $N$. Since $M \otimes -$ is functorial in $M$ by \cref{functoriality-M}, and since $\Hom(M, -)$ is consequently also functorial in $M$, we have that 
\begin{align*}
\Hom(M \otimes -, -) &= \lim \Hom(N \otimes -, -) \\
&= \lim \Hom(-, \Hom(N, -)) \\
&= \Hom(-, \lim \Hom(N, -)) \\
&= \Hom(-, \Hom(M, -)). \qedhere \end{align*}
\end{proof}

\begin{remark}
If $X$ happens to be a separately continuous $D(H)$-module, then $\epsilon$ and $\eta$ are continuous, so it is sufficient to show that $\epsilon$ and $\eta$ are both $E[H]$-linear. This can be done in a more ``hands-on'' way than the general case discussed above. Notice that $\epsilon(m \otimes \sigma) = \sigma(m)$, so 
\[ \begin{aligned} \epsilon( \delta_h \bdot (m \otimes \sigma)) &= \epsilon (h \bdot m \otimes \delta_h \bdot \sigma) \\
&= (\delta_h \bdot \sigma)(h \bdot m) \\
&= \delta_h \bdot \sigma(h^{-1} \bdot h \bdot m) \\
&= \delta_h \bdot \sigma(m) \\
&= \delta_h \bdot \epsilon(m \otimes \sigma), \end{aligned} \]
proving that $\epsilon$ is $E[H]$-linear. Similarly, since $\eta(x)(m) = m \otimes x$, we have 
\[ \begin{aligned} \eta( \delta_h \bdot x)(m) &= m \otimes \delta_h \bdot x \\
&= \delta_h \bdot (h^{-1} \bdot m \otimes x) \\
&= \delta_h \bdot \eta(x)(h^{-1} \bdot m) \\
&= (\delta_h \bdot \eta(x))(m) \end{aligned} \]
so $\eta$ is also $E[H]$-linear. 

\end{remark}

\subsection{Finite generation}

Our goal of this subsection is to prove that the module produced by \cref{construction} is sometimes finitely generated.  This statement is analogous to \cite[lemma 3.3]{ST01b}, and its proof largely follows the proof of that result. The precise statement is the following. 

\begin{proposition} \label{tensor-product-fd}
Suppose $H$ is compact. Let $M$ be a finite dimensional locally analytic representation of $H$, and let $X$ be finitely generated module over $D(H)$ which is annihilated by $\frh$. Then $M \otimes X$ is a finitely generated $D(H)$-module. 
\end{proposition}

Before proceeding with the proof, let us make a preliminary observation. 

\begin{remark} \label{two-actions-same}
Suppose $M$ is a finite dimensional locally analytic representation of $H$ (and $H$ need not be compact for the purposes of this remark). There are two ways to endow $M$ with the structure of a $D(H)$-module: 
\begin{enumerate}[(i)]
\item We can follow the construction of \cite[proposition 3.2 and the sentence before lemma 3.1]{ST02a}. 
\item We can note that the contragredient representation $M'$ is also locally analytic representation, so its dual $(M')'$ is a $D(H)$-module \cite[corollary 3.3]{ST02a}. We can then transfer this structure through the canonical isomorphism $M = (M')'$. 
\end{enumerate}
These two $D(H)$-module structures on $M$ coincide. Indeed, both $D(H)$-module structures are separately continuous, so it is sufficient to check that the $E[H]$-module structures coincide, but this follows from the fact that $M = (M')'$ is an isomorphism of representations of $H$.
\end{remark}

\begin{proof}[Proof of \cref{tensor-product-fd}]
\renewcommand{\bbH}{\mathbf{H}}
We remark that the assertion is true for $H$ if it is true for any open subgroup $H_1 \subseteq H$ (which is hence itself compact), since there is a canonical injective homomorphism $D(H_1) \to D(H)$. We may thus replace $H$ by an open subgroup in the course of the proof. 

After possibly shrinking $H$, we may assume that there is a good analytic open subgroup $\bbH$ of $H$ (in the sense of \cite[sec. 5.2]{EmertonA}) with the property that $H = \bbH^\circ(F)$ (cf. \cite[pp. 101-102]{EmertonA} for notation). We claim that the canonical map
\begin{equation}\label{injectivemap}
\begin{tikzcd} \cO(\bbH^\circ) \otimes_F C^\infty(H) \ar{r} & C^\an(H) \end{tikzcd} 
\end{equation} 
given by $\psi \otimes f \mapsto \psi|_H \cdot f$ is injective. 

\mbox{}

\noindent \textit{Proof of injectivity of \eqref{injectivemap}.} Suppose that for linearly independent elements $\psi_1, \dotsc, \psi_r \in \cO(\bbH^\circ)$ and functions $f_1, \ldots, f_r \in C^\infty(H)$ one has $\sum_{i=1}^r \psi_i|_H \cdot f_i = 0$. Choose a sufficiently small affinoid subgroup $\bbH_1 \sub \bbH^\circ$ such that $f_i|_{H_1h}$ is constant for all $i \in \{1, \dotsc, r\}$ and $h \in H$, where $H_1 = \bbH_1(F)$. If $c_{i,h}$ is the value of $f_i$ on the coset $H_1h$, then this implies that $\sum_{i=1}^r c_{i,h}\psi_i|_{H_1h} = 0$. Replacing $\psi_i$ by the function $x \mapsto \psi_i(xh)$, we may assume $h=1$. After choosing coordinates for $\bbH$ (coming from the $\mathscr{O}_F$-Lie lattice in $\frh$ which gives rise to $\bbH$, cf. \cite[p. 100]{EmertonA}), we may assume that $\bbH$ is, as a rigid analytic space, isomorphic to a $d$-dimensional polydisc of polyradius $(1, \dotsc,1)$ and $\bbH_1$ is a $d$-dimensional polydisc of polyradius $(|p|^n, \ldots, |p|^n)$ for some $n>0$. Then $\cO(\bbH^\circ)$ is the ring of power series $\sum_{\nu \in \bbN^d} a_\nu X^\nu \in F[[X_1, \ldots, X_d]]$ with $\lim_{|\nu|\ra \infty} |a_\nu|r^{|\nu|} =0$ for all $r<1$. If the restriction of such a power series to $(p^n\fro_F)^d$ vanishes, then all coefficients must vanish.\footnote{This can be seen by induction on $d$. We may assume without loss of generality that $f(X_1, \ldots, X_d)$ is a power series in $F\langle X_1, \ldots, X_d\rangle$ which vanishes on $\fro_F^d$. If $d=1$ then, as is well known, a non-zero power series in one variable can have only finitely many zeroes in any affinoid disc where it converges, hence $f=0$ in this case. If now $f = \sum_{k=0}^\infty f_k(X_2, \ldots,X_d)X_1^k$ is as above, then, for all fixed $(a_2, \ldots,a_d) \in \fro_F^{d-1}$, one has $f(X_1, a_2, \ldots,a_d) = 0$, by the case $d=1$. Hence $f_k(a_2, \ldots,a_d) = 0$ for all $k \ge 0$ and $(a_2, \ldots,a_d) \in \fro_F^{d-1}$. By induction, this implies that all $f_k$ vanish identically.} But this implies that $\sum_{i=1}^r c_{i,h}\psi_i = 0$, and hence $c_{i,h}=0$ for all $i \in \{1, \ldots,r\}$ and $h \in H$. \qed

\mbox{}

It suffices to prove that $M \otimes X$ is finitely generated when $X = D^\infty(H)$.\footnote{Since $X$ is a finitely generated $D(H)$-module which is annihilated by $\frh$, it is naturally a finitely generated module over $D^\infty(H)$. In other words, there is a $D(H)$-linear surjection $D^\infty(H)^{\oplus m} \to X$. Tensoring with $M$, we obtain another $D(H)$-linear surjection \[ (M \otimes D^\infty(H))^{\oplus m} = M \otimes \left(D^\infty(H)^{\oplus m} \right) \to M \otimes X. \] If $M \otimes D^\infty(H)$ is a finitely generated $D(H)$-module, then so is a finite direct sum of copies of this, and so is any quotient thereof.}
Write $\cD^\an(\bbH^\circ) = (\cO(\bbH^\circ)_E)'_b$ for the analytic distribution algebra of $\bbH^\circ$ with coefficients in $E$ (cf. \cite[2.2.2]{EmertonA} for the definition of $\cD^\an(\bbH^\circ)$, and the discussion of in \cite[p. 95-102]{EmertonA} for some properties of this algebra). This is a topological algebra whose underlying topological vector space is of compact type \cite[p. 102]{EmertonA}, and hence it is reflexive \cite[16.10]{NFA}. The canonical map $\cO(\bbH^\circ) \otimes_F E \to C^\an(H)$ gives rise to a continuous homorphism of topological algebras $D(H) \to \cD^\an(\bbH^\circ)$. 

Since $M$ is finite dimensional and locally analytic, the same is true for the dual space $M'$ with its contragredient action. There is thus a small enough group of the form $\bbH^\circ$ such that $M'$ is an analytic representation of $\bbH^\circ$ in the sense of \cite[3.6.1]{EmertonA}. As a consequence, $M = (M')'$ is naturally a module over $\cD^\an(\bbH^\circ)$, i.e., the map $D(H) \times M \to M$ which gives $M$ the structure of a $D(H)$-module factors as $D(H) \times M \to \cD^\an(\bbH^\circ) \times M$. Since $M$ is finite dimensional, we may argue by induction on the length of $M$ as a $\cD^\an(\bbH^\circ)$-module. It thus suffices to prove the case when $M$ is simple as $\cD^\an(\bbH^\circ)$-module. Choosing a generator for $M$ as a $\cD^\an(\bbH^\circ)$-module we obtain a surjection $\cD^\an(\bbH^\circ) \to M$ of $\cD^\an(\bbH^\circ)$-modules. This induces an $H$-equivariant injection $M' \hookrightarrow \cD^\an(\bbH^\circ)'_b = \cO(\bbH^\circ)_E$, by reflexivity. This map gives rise to an $H$-equivariant injection 
\[ \begin{tikzcd} M' \otimes C^\infty(G) \ar[hookrightarrow]{r} & \cO(\bbH^\circ)_E \otimes C^\infty(H) = \cO(\bbH^\circ) \otimes_F C^\infty(H) \end{tikzcd} \]
which we compose with the injection \eqref{injectivemap} to obtain an injective $H$-equivariant map
\begin{equation}\label{injectivemap2}
\begin{tikzcd} M' \otimes C^\infty(H) \ar[hookrightarrow]{r} & C^\an(H). \end{tikzcd}
\end{equation}
The group action of $H$ on the domain and target of \ref{injectivemap2} induces a derived action of $\frh$ and hence the structure of a $U(\frh)$-module on these representations. Since $M$ is finite-dimensional, the kernel $J$ of the map $U(\frh) \to \End(M')$ is a two-sided ideal of finite codimension.   Because the action of $\frh$ on $C^\infty(H)$ is trivial, the ideal $J$ also acts trivially on the domain of \ref{injectivemap2}. This implies that the image of \ref{injectivemap2} is contained in 
\[ C^\an(H)^{J=0} = \{f \in C^\an(H) \mid z \bdot f = 0 \text{ for all } z \in J\}. \]
By \cite[3.1]{ST01b}, the subspace topology on $C^\an(H)^{J=0}$ induced from $C^\an(H)$ is the finest locally convex topology. Because the natural topology on $M' \otimes C^\infty(H)$ is also the finest locally convex topology, we see that the map \ref{injectivemap2} is strict. As a consequence, the map $D(H) \to M \otimes D^\infty(H)$ obtained from \ref{injectivemap2} by passing to continuous dual spaces is surjective, cf. \cite[9.4]{NFA}.
\end{proof}
\renewcommand{\bbH}{\mathbb{H}}

%% file: modules.tex
\section{Modules over \texorpdfstring{$D(\frg,H)$}{D(g,H)}}

Let $G$ be a locally analytic group and $H$ a closed subgroup. In this section, we establish some foundational facts about the module theory of $D(\frg,H)$. Recall that this is defined to be the subring of $D(G)$ generated by $\frg$ and $D(H)$ \cite[section 3.4]{OrlikStrauchJH} \cite[section4]{ScSt}. 

\subsection{Some functional analysis}

\begin{lemma} \label{adjoint-continuity}
Suppose $\beta : V \times W \to Z$ is separately (resp. jointly) continuous bilinear map of locally convex vector spaces. Let $\beta^* : V \to \Hom(W, Z)$ be the adjoint map. Then $\beta^*(V) \subseteq \Hom^\cts(W, Z)$ and $\beta^*$ is continuous when $\Hom^\cts(W, Z)$ is given the topology of pointwise (resp. bounded) convergence. 
\end{lemma}

\begin{proof}
For fixed $v \in V$, note that $w \mapsto \beta(v, w)$ is continuous by separate continuity of $\beta$. Fix an open lattice $N$ in $Z$. If $w \in W$, continuity of $v \mapsto \beta(v, w)$ implies that there exists an open lattice $L$ in $V$ such that $\beta(L, w) \subseteq N$. In other words, $\beta^*(L) \subseteq L(w, N)$, so $\beta^*$ is continuous for the topology of pointwise convergence (ie, the weak topology \cite[p. 30]{NFA}). 

Now suppose $\beta$ is jointly continuous and $B$ is a bounded subset of $W$. Then there exist open lattices $L \subseteq V$ and $M \subseteq W$ such that $\beta(L, M) \subseteq N$. Since $B$ is bounded, there exists an $a \in E$ such that $B \subseteq aM$. Then \[ \beta(a^{-1}L, B) \subseteq \beta(a^{-1}L, aM) = \beta(L, M) \subseteq N, \]
so $\beta^*(a^{-1}L) \subseteq L(B, M)$. This shows that $\beta^*$ is continuous for the topology of bounded convergence (ie, the strong topology \cite[p. 30]{NFA}). 
\end{proof}

The following is a version of \cite[proposition 1.2.28]{EmertonA} in which the space is not necessarily compact, but the vector space is finite dimensional. 

\begin{lemma} \label{functions-into-fd}
If $V$ is a finite dimensional vector space, then the natural map $C^\an(H) \otimes V \to C^\an(H, V)$ given by $f \otimes v \mapsto [h \mapsto f(h)v]$ is a topological isomorphism. 
\end{lemma}

\begin{proof}
The bilinear map $(f, v) \mapsto [h \mapsto f(h)v]$ is separately continuous, so the induced map $C^\an(H) \otimes V \to C^\an(H, V)$ is continuous. If $v_1, \dotsc, v_n$ is a basis for $V$, and if $f \in C^\an(H, V)$, then we can define functions $f_i \in C^\an(H)$ by
\[ f(h) = \sum_i f_i(h) v_i. \]
The map $C^\an(H, V) \to C^\an(H) \otimes V$ given by $f \mapsto \sum f_i \otimes v_i$ defines a continuous inverse to the map in the statement. 
\end{proof}

\begin{definition} \label{convolution}
Any function $f \in C^\an(H)$ defines a continuous  multiplication map $C^\an(H) \to C^\an(H)$. It thus induces a continuous map on dual spaces $D(H) \to D(H)$, which we denote by $\mu \mapsto \mu_f$. Explicitly, for $g \in C^\an(H)$, we have 
\[ \mu_f(g) = \mu(fg). \]
Note that the map $(f, \mu) \mapsto \mu_f$ is bilinear and separately continuous (see below). Thus the induced map $C^\an(H) \otimes_i D(H) \to D(H)$ is also continuous. 
\end{definition}

\begin{proof}[Proof of separate continuity]
We have already noted that $(f, \mu) \mapsto \mu_f$ is continuous for fixed $f$. For fixed $\mu$, let $M$ be an open lattice in $E$, and let $B$ is a bounded subset of $C^\an(H)$, so that \[ L(B, M) = \{ \mu \in D(H) : \mu(B) \subseteq M\} \] is an open lattice in $D(H)$. We want to construct an open lattice $L$ in $C^\an(H)$ such that $\mu_f \in L(B,M)$ for all $f \in L$. 

Since the multiplication map $C^\an(H) \otimes_\pi C^\an(H) \to C^\an(H)$ and $\mu : C^\an(H) \to E$ are both continuous, there exist lattices $L_1$ and $L_2$ in $C^\an(H)$ such that \[ \mu(L_1 L_2) \subseteq M. \] Since $B$ is bounded, there exists an $a \in E$ such that $B \subseteq aL_2$. Let $L = a^{-1}L_1$ and note that this an open lattice in $C^\an(H)$. Then any $f \in L$ can be written as $a^{-1}f_1$ for $f_1 \in L_1$. Similarly, since $B \subseteq aL_2$, any $g \in B$ can be written as $ag_2$ where $g_2 \in L_2$. Then 
\[ \mu_f(g) = \mu(fg) = \mu(a^{-1}f_1 a g_2) = \mu(f_1 f_2) \in \mu(L_1 L_2) \subseteq M \]
which shows that $\mu_f \in L(B, M)$ for all $f \in L$. 
\end{proof}

\subsection{The adjoint-twisted braiding}

\begin{definition} \label{zeta-definition}
Let $\zeta$ denote the composite 
\[ \begin{tikzcd} (D(H) \otimes \frg) \otimes_i C^\an(H, \End(\frg))  \ar{r}{\sim} & (D(H) \otimes \frg) \otimes_i (C^\an(H) \otimes \End(\frg)) \ar{d} \\ & \frg \otimes D(H) \end{tikzcd} \]
where the horizontal map is the tensor product of $D(H) \otimes \frg$ with the topological isomorphism from \cref{functions-into-fd} (applied with $V = \End(\frg)$), and the vertical one is given by \[ (\mu \otimes x) \otimes (f \otimes \sigma) \mapsto \sigma(x) \otimes \mu_f. \]
The evaluation map $\frg \otimes \End(\frg) \to \frg$ is automatically continuous (as it is a linear map between finite dimensional vector spaces), and we have already noted that $C^\an(H) \otimes_i D(H) \to D(H)$ is also continuous in \cref{convolution}. It follows that $\zeta$ is continuous. 
\end{definition}

\begin{definition}
Let $\Ad \in C^\an(H, \End(\frg))$ denote the adjoint action of $H$ on $\frg$. We define
\[ \begin{tikzcd} D(H) \otimes \frg \ar{r}{\zeta_{\Ad}} & \frg \otimes D(H) \end{tikzcd} \] 
to be the map given by $\mu \otimes x \mapsto \zeta(\mu \otimes x \otimes \Ad)$. Continuity of $\zeta$ implies that $\zeta_{\Ad}$ is continuous. We call this map the \emph{adjoint-twisted  braiding} of $D(H)$ and $\frg$. 
\end{definition}

\begin{lemma} \label{reformulate-joho-lemma}
Following \cite[proof of lemma 3.5]{OrlikStrauchJH}, suppose $x_1, \dotsc, x_n$ is a basis for $\frg$, and fix $x \in \frg$. Let $c_i \in C^\an(H)$ be defined by $\Ad(h)(x) = \sum_i c_i(h) x_i$ and let $\mu_i = \mu_{c_i}$. Then 
\[ \zeta_{\Ad}(\mu \otimes x) = \sum_i x_i \otimes \mu_i. \]
\end{lemma}

\begin{proof}
This is a straightforward calculation. Let $\sigma_{i,j} \in \End(\frg)$ be the map that sends $x_j$ to $x_i$ and annihilates all $x_k$ for $k \neq j$. Under the isomorphism of \cref{functions-into-fd}, $\Ad \in C^\an(H, \End(\frg))$ corresponds to $\sum_{i,j} c_{i,j} \otimes \sigma_{i,j}$ where $c_{i,j} \in C^\an(H)$ is defined by
\[ \Ad(h) = \sum_{i,j} c_{i,j}(h) \sigma_{i,j}. \]
Then
\[ \zeta_{\Ad}(\mu \otimes x) = \sum_{i,j} \sigma_{i,j}(x) \otimes \mu_{c_{i,j}}. \]
Suppose $x = \sum_j b_j x_j$ for $b_j \in E$. Then $\sigma_{i,j}(x) = b_j x_i$, so 
\[ \zeta_{\Ad}(\mu \otimes x) = \sum_{i,j} b_j x_i \otimes \mu_{c_{i,j}} = \sum_i x_i \otimes \left( \sum_j b_j \mu_{c_{i,j}} \right). \]
Now note that
\[ \sum_i c_i(h) x_i = \Ad(h)(x) = \sum_i \left( \sum_j c_{i,j}(h) b_j \right) x_i \]
so $c_i = \sum b_j c_{i,j}$. Thus
\[ \left( \sum b_j \mu_{c_{i,j}} \right)(g) = \sum_j \mu(b_j c_{i,j} g) = \mu(c_i g) = \mu_{c_i}(g) = \mu_i(g), \]
so $\zeta_{\Ad}(\mu \otimes x) = \sum_i x_i \otimes \mu_i$. 
\end{proof}

\begin{corollary} \label{joho-and-adjoint}
Suppose $x \in \frg$. If $h \in H$ and $\delta_h \in D(H)$ is the corresponding delta distribution, then 
\[ \zeta_{\Ad}(\delta_h \otimes x) = \Ad(h)(x) \otimes \delta_h. \]
\end{corollary}

\begin{proof}
As in \cref{reformulate-joho-lemma}, let $x_1, \dotsc, x_n$ be a basis for $\frg$ and let $c_i \in C^\an(H)$ be defined by $\Ad(h)(x) = \sum_i c_i(h) x_i$. For a distribution $\mu \in D(H)$, set $\mu_i = \mu_{c_i}$. Then
\[ \Ad(h)(x) \otimes \delta_h = \sum_i x_i \otimes c_i(h)\delta_h, \]
but 
\[ (c_i(h)\delta_h)(g) = c_i(h) g(h) = \delta_{h,i}(g) \]
and then \cref{reformulate-joho-lemma} allows us to conclude that 
\[ \Ad(h)(x) = \sum x_i \otimes \delta_{h, i} = \zeta_{Ad}(\delta_h \otimes x). \qedhere \]
\end{proof}

\begin{corollary} \label{is-in-category}
The following diagram of continuous maps commutes. 
\[ \begin{tikzcd} D(H) \otimes \frg \ar{r}{\zeta_{\Ad}} \ar[bend right]{dr} & \frg \otimes D(H) \ar{d} \\ & D(G) \end{tikzcd} \]
Here, both maps into $D(G)$ are restrictions of the multiplication map on $D(G)$. 
\end{corollary}

\begin{proof}
The commutativity follows from \cref{reformulate-joho-lemma} and \cite[proof of lemma 3.5]{OrlikStrauchJH}. The fact that the maps into $D(G)$ are continuous is a consequence of the fact that the multiplication map on $D(G)$ is separately continuous \cite[proposition 2.3]{ST02b}. 
\end{proof}

\subsection{Universal property of \texorpdfstring{$D(\frg,H)$}{D(g,H)}}

\begin{definition}
Let $\mathscr{C}$ denote the category of triples $(A, \alpha, \beta)$ where $A$ is a $E$-algebra, $\alpha : D(H) \to A$ is a homomorphism of $E$-algebras, and $\beta : \frg \to A$ is a homomorphism of Lie algebras, such that $\alpha|_{\frh} = \beta|_{\frh}$ and such that the following diagram commutes, where $\zeta_{\Ad}$ is the adjoint-twisted braiding. 
\[ \begin{tikzcd} D(H) \otimes \frg \ar{r}{\zeta_{\Ad}} \ar[bend right]{dr}[swap]{\alpha \otimes \beta} & \frg \otimes D(H) \ar{d}{\beta \otimes \alpha} \\ & A \end{tikzcd} \]
Morphisms $(A, \alpha, \beta) \to (A', \alpha', \beta')$ in $\mathscr{C}$ are $E$-algebra homomorphisms $A \to A'$ commuting with the maps from $D(H)$ and $\frg$.
\[ \begin{tikzcd} D(H) \ar{r}{\alpha} \ar[bend right]{dr}[swap]{\alpha'} & A \ar{d} & \frg \ar{r}{\beta} \ar[bend right]{dr}[swap]{\beta'} & A \ar{d} \\
& A' & & A' \end{tikzcd} \]
\end{definition}

\begin{remark} \label{restrict}
Suppose $(A, \alpha, \beta) \in \mathscr{C}$. If $A'$ is the $E$-subalgebra of $A$ generated by the images of $\alpha$ and $\beta$, then evidently $(A', \alpha, \beta)$ is again an object of $\mathscr{C}$. 
\end{remark}

\begin{theorem} \label{universal-prop}
$D(\frg,H)$ is an initial object of $\mathscr{C}$. 
\end{theorem}

\begin{proof}
It follows from \cref{is-in-category} that $D(G)$ is naturally in $\mathscr{C}$, and then \cref{restrict} implies that $D(\frg,H)$ is also naturally in $\mathscr{C}$. Let $(R, \alpha, \beta)$ be an initial object of $\mathscr{C}$ (which exists, for instance by \cite[theorem 9.4.14]{BergmanUniversal}) and consider the induced map $\sigma : R \to D(\frg,H)$ in $\mathscr{C}$. We will show that $\sigma$ is an isomorphism.

Since $\beta$ is a homomorphism of Lie algebras and since $\alpha|_{\frh} = \beta|_{\frh}$, the map $\beta \otimes \alpha$ induces a well-defined $U(\frg)$-$D(H)$-bimodule homomorphism $\gamma : U(\frg) \otimes_{U(\frh)} \to R$ such that the following diagram commutes. 
\[ \begin{tikzcd} U(\frg) \otimes_{U(\frh)} D(H) \ar{r}{\gamma} \ar[bend right]{dr}[swap]{\sim} & R \ar{d}{\sigma} \\ & D(\frg,H) \end{tikzcd} \]
We know that $U(\frg) \otimes_{U(\frh)} D(H) \to D(\frg,H)$ is an isomorphism \cite[section 4]{ScSt}, so $\sigma$ is surjective. To show that $\sigma$ is injective, it is sufficient to show that $\gamma$ is surjective. 

Observe that, since $R$ is initial in $\mathscr{C}$, it must be generated as an algebra by the images of $\alpha$ and $\beta$ (cf. \cref{restrict}). Using the fact that $\alpha$ is an algebra  homomorphism, this implies that $R$ is generated as a right $D(H)$-module by expressions of the form 
\begin{equation} \label{element-of-universal} \alpha(\mu_1) \beta(x_1) \alpha(\mu_2) \beta(x_2) \dotsb \alpha(\mu_n) \beta(x_n) \end{equation}
where $n \geq 0$, $\mu_1, \dotsc, \mu_{n} \in D(H)$, and $x_1, \dotsc, x_n \in \frg$. Since $\gamma$ is right $D(H)$-linear, it is sufficient to show that every element of the form \eqref{element-of-universal} is in the image of $\gamma$. To show this, we induct on $n$ and use the fact that the following diagram commutes. 
\[ \begin{tikzcd} D(H) \otimes \frg \ar{r}{\zeta_{\Ad}} \ar[bend right]{dr}[swap]{\alpha \otimes \beta} & \frg \otimes D(H) \ar{d}{\beta \otimes \alpha} \\ & R \end{tikzcd} \]
This tells us that 
\[ \alpha(\mu_1)\beta(x_1) = \sum_i \beta(x_{1,i}) \alpha(\mu_{1,i}) \]
for some $x_{1,i} \in \frg$ and $\mu_{1,i} \in D(H)$, so
\[ \eqref{element-of-universal} = \sum_i \beta(x_{1,i}) \alpha(\mu_{1,i} \mu_2) \beta(x_2) \dotsb \alpha(\mu_n) \beta(x_n), \]
where we have again used the fact that $\alpha$ is a homomorphism of $E$-algebras. By induction, we know that $\alpha(\mu_{1,i} \mu_2) \beta(x_2) \dotsb \alpha(\mu_n) \beta(x_n)$ is in the image of $\gamma$. Since $\gamma$ is left $U(\frg)$-linear, it follows that \eqref{element-of-universal} is also in the image of $\gamma$. 
\end{proof}

\begin{corollary} \label{DgP-module-structure}
The data of a $D(\frg,H)$-module structure on a vector space $M$ is precisely the data of a $D(H)$-module structure and a $\frg$-module structure such that the two $\frh$-module structures agree and such that, for any $\mu \in D(H)$ and $x \in \frg$ and $m \in M$, 
\[ \zeta_{\Ad}(\mu \otimes x) = \sum x_i \otimes \mu_i \implies \mu \bdot (x \bdot m) = \sum_i x_i \bdot (\mu_i \bdot m). \]
Moreover, a map between two $D(\frg,H)$-modules is $D(\frg,H)$-linear if and only if it is both $D(H)$-linear and $\frg$-linear. 
\end{corollary}

\begin{remark} \label{continuity-DgP-prelim}
The group ring $E[H]$ is dense in $D(H)$, so $E[H] \otimes \frg$ is dense in $D(H) \otimes \frg$. Suppose $(A, \alpha, \beta)$ is a triple consisting of a separately continuous $E$-algebra, a continuous $E$-algebra homomorphism $D(H) \to A$, and a continuous Lie algebra homomorphism $\frg \to A$ such that $\alpha|_{\frh} = \beta|_{\frh}$. Since multiplication on $A$ is separately continuous, the two maps $D(H) \otimes \frg \to A$ and $\frg \otimes D(H) \to A$ are continuous. Thus $(\beta \otimes \alpha) \circ \zeta_{\Ad} = \alpha \otimes \beta$ if and only if this is true when restricted to $E[H] \otimes \frg$. 
\[ \begin{tikzcd} E[H] \otimes \frg \ar[hookrightarrow]{r} & D(H) \otimes \frg \ar{r}{\zeta_{\Ad}} \ar[bend right]{dr}[swap]{\alpha \otimes \beta} & \frg \otimes D(H) \ar{d}{\beta \otimes \alpha} \\ 
& & A \end{tikzcd} \]
In light of \cref{joho-and-adjoint}, this means that $(A, \alpha, \beta) \in \mathscr{C}$ if and only if 
\[ \alpha(\delta_h) \beta(x) = \beta(\Ad(h)(x))\alpha(\delta_h) \]
for all $h \in H$ and $x \in \frg$. 
\end{remark}

\begin{lemma} \label{continuity-DgP}
Suppose $M$ is simultaneously a separately continuous $\frg$-module and a separately continuous $D(H)$-module such that the two actions of $\frh$ agree. Then $M$ is a $D(\frg,H)$-module if and only if 
\[ \delta_h \bdot (x \bdot m) = \Ad(h)(x) \bdot (\delta_h \bdot m) \]
for all $h \in H$, $x \in \frg$, and $m \in M$. 
\end{lemma}

\begin{proof}
Giving $\End^\cts(M)$ the topology of pointwise convergence, we have a continuous map $D(H) \to \End^\cts(M)$ by \cref{adjoint-continuity}. Thus we can  apply \cref{continuity-DgP-prelim} with $A = \End^\cts(M)$. 
\end{proof}

\begin{lemma} \label{quotients-isomorphic}
Let $D^\infty(H)$ denote the quotient of $D(H)$ by the two-sided ideal generated by $\frh$. There exists a natural surjective homomorphism $D(\frg,H) \to D^\infty(H)$ of $E$-algebras whose kernel is the two-sided ideal of $D(\frg,H)$ generated by $\frg$. 
\end{lemma}

\begin{proof}
Observe that $D^\infty(H)$ becomes an object of $\mathscr{C}$ when equipped with the quotient map $D(H) \to D^\infty(H)$ and the zero map $\frg \to D^\infty(H)$. Thus there exists a natural $E$-algebra homomorphism  $\sigma : D(\frg,H) \to D^\infty(H)$, which is clearly surjective. It is also clear that its kernel contains $\frg$. Letting $J(\frg)$ denote the two-sided ideal generated by $\frg$, we see that there is an induced map $\bar{\sigma} : D(\frg,H)/J(\frg) \to D^\infty(H)$. Showing that $\ker(\sigma) = J(\frg)$ is equivalent to showing that $\bar{\sigma}$ is injective. 

We have a map $\tau : D(H) \to D(\frg,H)/J(\frg)$ which annihilates $\frh$, so it induces a map $\bar{\tau} : D^\infty(H) \to D(\frg,H)/J(\frg)$. It is clear that $\bar{\sigma} \circ \bar{\tau}$ is the identity on $D^\infty(H)$. Thus $\bar{\sigma}$ is injective if and only if $\bar{\tau}$ is surjective. And $\bar{\tau}$ is surjective if and only if $\tau$ is surjective, so we show that $\tau$ is surjective. 

As we noted in the proof of \cref{universal-prop}, every element of $D(\frg,H)$ can be written as a finite sum of expressions of the form 
\[ \mu_1 x_1 \mu_2 x_2 \dotsb \mu_n x_n \mu_{n+1} \]
for $n \geq 0$, $\mu_1, \dotsc, \mu_{n+1} \in D(H)$, and $x_1, \dotsc, x_n \in \frg$. But all such expressions are in $J(\frg)$ if $n \geq 1$, and the above expression is just an element of $D(H)$ when $n = 0$. Thus any element of $D(\frg,H)$ is congruent to an element of $D(H)$ modulo $J(\frg)$. This shows that $\tau$ is surjective. 
\end{proof}

\subsection{Tensor-hom adjunction for \texorpdfstring{$D(\frg, H)$}{D(g,H)}-modules}

The following is a locally analytic analog of a  definition made in \cite[part II, section 1.20]{Jantzen}. 

\begin{definition} \label{gP-module}
A \emph{locally analytic $(\frg,H)$-module} $M$ is a locally analytic representation of $H$ which is also simultaneously a separately continuous $\frg$-module, such that these two actions satisfy two compatibility conditions: 
\begin{enumerate}[(C1)]
\item The two induced actions of $\frh$ agree, and 
\item $h \bdot (x \bdot m) = \Ad(h)(x) \bdot (h \bdot m)$ for all $h \in H, x \in \frg$ and $m \in M$. 
\end{enumerate}
\end{definition}

\begin{lemma} \label{gP-module-DgP-module}
A locally analytic $(\frg,H)$ module is naturally a $D(\frg,H)$-module. 
\end{lemma}

\begin{proof}
Any locally analytic representation of $H$ can be given the structure of a separately continuous $D(H)$-module \cite[proposition 3.2]{ST02b}, so this follows from \cref{continuity-DgP}
\end{proof} 

\begin{lemma} \label{tensor-DgP-module-construction}
Suppose $M$ is a locally analytic $(\frg,H)$-module which is locally finite dimensional as a representation of $H$. If $X$ is a $D(\frg,H)$-module, then $M \otimes X$ is also a $D(\frg,H)$-module. This construction is functorial in both $M$ and $X$. 
\end{lemma}

\begin{proof}
Suppose that $M$ is finite dimensional. Then $M \otimes X$ is naturally a $\frg$-module, and it is also a $D(H)$-module by \cref{construction}. \Cref{restrict-tensor-product} implies that the two $\frh$-module structures coincide. We want to show that the following commutes. 
\[ \begin{tikzcd} D(H) \otimes \frg \ar[bend right]{dr} \ar{r}{\zeta_{\Ad}} & \frg \otimes D(H) \ar{d} \\ & \End(M \otimes X) \end{tikzcd} \]
For brevity, let us define
\[ \begin{aligned}
L &= \left( \left(D(H) \otimes \frg \right) \,\hat{\otimes}_i \, D(H) \right) \oplus  \left( D(H) \, \hat{\otimes}_i \left( D(H) \otimes \frg \right) \right) \\
R &=\left( \frg \otimes \left(D(H) \right) \,\hat{\otimes}_i \, D(H) \right) \oplus  \left( D(H) \, \hat{\otimes}_i \left( \frg \otimes D(H) \right) \right)
\end{aligned} \]
and observe that there is a map $\zeta : L \to R$ given by \[ \zeta = (\zeta_\Ad \otimes D(H)) \oplus (D(H) \oplus \zeta_{\Ad}).  \]
Note that we can decompose the above diagram as follows. 
\[ \begin{tikzcd} 
D(H) \otimes \frg \ar{r}{\zeta_\Ad} \ar{d} & \frg \otimes D(H) \ar{d} \\
D(H \times H) \otimes (\frg \oplus \frg) \ar{r}{\zeta_{\Ad}} \ar[equals]{d} & (\frg \oplus \frg) \otimes D(H \times H) \ar[equals]{d} \\
\left(D(H) \,\hat{\otimes}_i \, D(H) \right) \otimes (\frg \oplus \frg) \ar[equals]{d}  & (\frg \oplus \frg) \otimes \left(D(H) \,\hat{\otimes}_i \, D(H) \right) \ar[equals]{d} \\ 
L \ar[bend right]{dr} \ar{r}{\zeta} &  R \ar{d} \\
& \End(M \otimes X) \end{tikzcd} \]
It is straightforward to verify that the triangle at the bottom commutes because $X$ and $M$ are both individually $D(\frg,H)$-modules. 

The square at the very top commutes since $\zeta_{\Ad}$ is functorial in the pair $(G, H)$ and we can apply this functoriality to the diagonal map $(G, H) \to (G \times G, H \times H)$, but we won't need to use this commutativity. We only note that the maps $\frg \to (\frg \oplus \frg)$ involved in this square are the diagonal map $x \mapsto (x, x)$. We will show that the following square commutes. 
\[ \begin{tikzcd} 
D(H) \otimes \frg \ar{r}{\zeta_\Ad} \ar{d} & \frg \otimes D(H) \ar{d} \\
L \ar{r}{\zeta} & R \end{tikzcd} \]
All of these maps are continuous, so we can restrict to check commutativity on $E[H] \otimes \frg$. If $h \in H$ and $\delta_h$ is the corresponding delta distribution and $x \in \frg$, note that the map into $L$ carries $\delta_h \otimes x$ to 
\[ \delta_h \otimes x \otimes \delta_h + \delta_h \otimes \delta_h \otimes x, \]
and similarly with the map into $R$. Using the formula for $\zeta_\Ad$ from \cref{joho-and-adjoint}, commutativity follows. 
\end{proof}

\begin{corollary}
Suppose $M$ is a locally analytic $(\frg,H)$-module which is locally finite dimensional as a representation of $H$. Then $\Hom(M, X)$ is naturally a $D(\frg,H)$-module for any $D(\frg,H)$-module $X$. 
\end{corollary}

\begin{proof}
We proceed just as in \cref{hom-construction}. 
\end{proof}

\begin{theorem} \label{DgP-adjunction}
Suppose $M$ is a locally analytic $(\frg,H)$-module which is locally finite dimensional as a representation of $H$. Then there is an adjunction of $D(\frg,H)$-module-valued functors $(M \otimes -) \dashv \Hom(M, -)$. 
\end{theorem}

\begin{proof}
It is sufficient to show that the unit and counit of the adjunction are $D(\frg,H)$-linear. But we know that $(M \otimes -) \dashv \Hom(M, -)$ as $\frg$-module-valued functors and as $D(H)$-module functors (by \cref{adjunction}), so the counit and unit are both $\frg$ and $D(H)$-linear. This implies $D(\frg,H)$-linearity. 
\end{proof}

\subsection{The closure \texorpdfstring{$D(G)_H$}{D(G)H} of \texorpdfstring{$D(\frg,H)$}{D(g,H)}}

Let us now consider Kohlhaase's ring $D(G)_H$ of distributions supported in $H$ \cite[1.2.1--6]{KohlhaaseI}. This is precisely the topological closure of $D(\frg,H)$ inside $D(G)$ \cite[1.2.10]{KohlhaaseI}. In this section, we prove that, if $G$ is compact, then $D(G)_H$ is a Fr\'echet-Stein algebra. Let $D_r(G)_H$ denote the closure of $D(G)_H$ inside $D_r(G)$ \cite[p. 30]{KohlhaaseI}. In other words, $D_r(G)_H$ is the completion of $D(G)_H$ for the topology defined by the $r$-norm. 

\begin{lemma} \label{supported-distributions-noetherian}
If $G$ is compact, then $D_r(G)_H$ is left noetherian. 
\end{lemma}

\begin{proof}
We know that $U_r(\frg)$ is noetherian, and $D_r(G)_H$ is finite free as a module over $U_r(\frg)$ \cite[theorem 1.4.2]{KohlhaaseI}, so $D_r(G)_H$ is also noetherian. 
\end{proof}

\begin{lemma} \label{distributions-finite-free-over-supported}
If $G$ is compact, then $D_r(G)$ is finite free as a right module over $D_r(G)_H$. 
\end{lemma}

\begin{proof}
Choose topological generators $a_1, \dotsc, a_d$ of $G$ such that $a_{k+1}, \dotsc, a_d$ are topological generators for $H$. Set $b_i = a_i - 1$. Then there exists an $\ell_i$ such that $D_r(G)$ is finite free as a right module over $U_r(\frg)$ with basis given by monomials $b^\alpha = b_1^{\alpha_1} \dotsb b_d^{\alpha_d}$ where $\alpha_i \leq \ell_i$ for all $i$ \cite[theorem 1.4.2]{KohlhaaseI}. Furthermore, the subring $D_r(G)_H$ is also finite free as a right module over $U_r(\frg)$, with basis given by monomials of the form $b_{k+1}^{\alpha_{k+1}} \dotsb b_d^{\alpha_d}$ where $\alpha_i \leq \ell_i$ for all $i$ \cite[corollary 1.4.3]{KohlhaaseI}. 

This implies that $D_r(G)$ is free over $D_r(G)_H$ with basis given by monomials of the form $b_{1}^{\alpha_{1}} \dotsb b_k^{\alpha_k}$ where $\alpha_{i} \leq \ell_i$ for all $i$. To see this, first observe that any element of $D_r(G)$ can be written in the form $\sum_\alpha b^\alpha c_\alpha$ where $c_{\alpha} \in U_r(\frg)$. But 
\[ \sum_\alpha b^\alpha c_\alpha = \sum_{\alpha} b_1^{\alpha_1} \dotsb b_d^{\alpha_d}c_{\alpha}  = \sum_{\alpha_{1}, \dotsc, \alpha_k} b_1^{\alpha_1} \dotsb b_k^{\alpha_k} \left( \sum_{\alpha_{k+1}, \dotsc, \alpha_d} b_{k+1}^{\alpha_{k+1}} \dotsb b_d^{\alpha_d} c_\alpha \right) \]
and
\[ \sum_{\alpha_{k+1}, \dotsc, \alpha_d} b_{k+1}^{\alpha_{k+1}} \dotsb b_d^{\alpha_d} c_\alpha \]
is an element of $D_r(G)_H$, which shows that $D_r(G)$ is generated as a right module over $D_r(G)_H$ by monomials of the form $b_{1}^{\alpha_{1}} \dotsb b_k^{\alpha_k}$. 
		
Moreover, these monomials are linearly independent over $D_r(G)_H$. Indeed, suppose we have a dependence relation 
\[ 0 = \sum_{\alpha_{1}, \dotsc, \alpha_k} b_1^{\alpha_1} \dotsb b_k^{\alpha_k} s_{\alpha_1, \dotsc, \alpha_k}  \]
where $s_{\alpha_1, \dotsc, \alpha_k} \in D_r(G)_H$. Each $s_{\alpha_1, \dotsc, \alpha_k}$ can be written as 
\[ s_{\alpha_1, \dotsc, \alpha_k} = \sum_{\alpha_{k+1}, \dotsc, \alpha_d} b_{k+1}^{\alpha_{k+1}} \dotsb b_d^{\alpha_d} c_\alpha \]
where $c_\alpha \in U_r(\frg)$. Thus we can rewrite the dependence relation as 
\[ \begin{aligned} 0 &= \sum_{\alpha_{1}, \dotsc, \alpha_k} b_1^{\alpha_1} \dotsb b_k^{\alpha_k} s_{\alpha_1, \dotsc, \alpha_k} \\
&= \sum_{\alpha_1, \dotsc, \alpha_k} b_1^{\alpha_1} \dotsb b_k^{\alpha_k} \left( \sum_{\alpha_{k+1}, \dotsc, \alpha_d} b_{k+1}^{\alpha_{k+1}} \dotsb b_d^{\alpha_d} c_\alpha \right) \\
&= \sum_\alpha b_1^{\alpha_1} \dotsb b_d^{\alpha_d} c_\alpha \\
&= \sum_\alpha b^\alpha c_\alpha. \end{aligned} \]
Since the monomials $b^\alpha$ are linearly independent over $U_r(\frg)$, this means that $c_\alpha = 0$ for all $\alpha$, so $s_{\alpha_1, \dotsc, \alpha_k} = 0$ for all $\alpha_1, \dotsc, \alpha_k$ as well. 
\end{proof}

\begin{proposition} \label{supported-distributions-frechet-stein}
If $G$ is compact, then $D(G)_H$ is a Fr\'echet-Stein algebra. 
\end{proposition}

\begin{proof}
Certainly $D(G)_H$, as a closed subalgebra of $D(G)$, is Fr\'echet, and we already know that $D_r(G)_H$ is noetherian by \cref{supported-distributions-noetherian}, so we only need to check that the transition maps $D_r(G)_H \to D_{r'}(G)_H$ are right flat. Observe that we have a commutative square as follows. 
\[ \begin{tikzcd} D_r(G)_H \ar{r} \ar{d} & D_{r'}(G)_H \ar{d} \\ D_r(G) \ar{r} & D_{r'}(G) \end{tikzcd} \]
Since $D(G)$ is Fr\'echet-Stein, we know that the horizontal map $D_r(G) \to D_{r'}(G)$ is right flat. We also saw above in \cref{distributions-finite-free-over-supported} that both vertical maps are faithfully flat. This means that $D_{r'}(G)$ is right flat as a $D_r(G)_H$-module and faithfully flat as a $D_{r'}(G)_H$-module, so it follows that $D_r(G)_H \to D_{r'}(G)_H$ is right flat \cite[tag 039V]{stacks-project}.
\end{proof}

\subsection{c-Flatness}

Recall that, if $A \to B$ is a continuous homomorphism of Fr\'echet-Stein algebras, there is a right-exact functor $B \arcotimes_A \,-$ from coadmissible left $A$-modules to coadmissible left $B$-modules \cite[section 7]{AWDcapI}. One says that $B$ is \emph{right c-flat} if this functor is exact, and \emph{right faithfully c-flat} if this functor is faithfully exact. If $M$ is a finitely presented $A$-module, then $B \arcotimes_A\, M = B \otimes_A M$. 

\begin{lemma} \label{checking-flatness-after-completion}
Suppose $A \to B$ is a continuous homomorphism of Fr\'echet-Stein algebras, and let $(q_n)_n$ and $(p_n)_n$ be sequences of norms defining the topologies on $A$ and $B$, respectively, such that $(A, q_n) \to (B, p_n)$ is continuous for all $n$ (cf. \cite[proof of lemma 3.8]{ST03}). 
\begin{enumerate}[(a)]
\item If $A_{q_n} \to B_{p_n}$ is right flat for all $n$, then $A \to B$ is right c-flat. 
\item If $A_{q_n} \to B_{p_n}$ is right faithfully flat for all $n$, then $A \to B$ is right faithfully c-flat. 
\end{enumerate}
\end{lemma}
	
\begin{proof}
Suppose $M \to N$ is an injective map of coadmissible left $A$-modules. Since $A_{q_n}$ is right flat over $A$ \cite[remark 3.2]{ST03}, and $B_{p_n}$ is right flat over $A_{q_n}$,we see that $B_{p_n} \otimes_A M \to B_{p_n} \otimes_A N$ is also injective. Since limits are left exact, we conclude that 
\[ \begin{tikzcd} B \arcotimes_A \,M = \lim B_{p_n} \otimes_A M \ar{r} & \lim B_{p_n} \otimes_A N = B \arcotimes_A \,N \end{tikzcd} \]
is injective, proving that $B$ is right c-flat. 

Now suppose $A_{q_n} \to B_{p_n}$ is right faithfully flat for all $n$ and that $M$ is a coadmissible left $A$-module such that $B \arcotimes_A M = 0$. Since the image of $B \arcotimes_A M$ is dense in $B_{p_n} \otimes_A M$ \cite[theorem A and corollary 3.1]{ST03}, this means that \[ B_{p_n} \otimes_{A_{q_n}} (A_{q_n} \otimes_A M) = B_{p_n} \otimes_A M = 0 \]
for all $n$. Since $B_{p_n}$ is faithfully flat over $A_{q_n}$, we see that $A_{q_n} \otimes_A M = 0$ for all $n$, so $M = \lim A_{q_n} \otimes_A M = 0$ as well. 
\end{proof}

\begin{lemma} \label{distributions-flat-over-supported-distributions}
Suppose $G$ is compact. Then $D(G)$ is faithfully c-flat as a right module over $D(G)_{H}$. 
\end{lemma}
	
\begin{proof}
This follows from \cref{checking-flatness-after-completion,supported-distributions-frechet-stein,distributions-finite-free-over-supported}. 
\end{proof}

\begin{remark}
If $D(G)_H$ is a left coherent ring, it would also be true that $D(G)$ is right flat over $D(G)_H$. Indeed, then any finitely generated left ideal $I$ in $D(G)_H$ would also be finitely presented, so 
\[ D(G) \otimes_{D(G)_H} I = D(G) \arcotimes_{D(G)_H} \,I. \]
Furthermore, since finite presentation implies coadmissibility \cite[corollary 3.4(v)]{ST03}, we would also know that $D(G) \otimes_{D(G)_H} I = D(G) \arcotimes_{D(G)_H} \,I \to D(G)$ is injective by \cref{distributions-flat-over-supported-distributions}. This would prove that $D(G)$ is right flat over $D(G)_H$ \cite[4.12]{Lam}. However, we do not know if $D(G)_H$ is a coherent ring. 
\end{remark}

\subsection{Coadmissible modules over \texorpdfstring{$D(G)_H$}{D(G)H}}

\begin{lemma}[{\cite[equation (1.7)]{KohlhaaseI}}] \label{compact-direct-sum}
Suppose $G_0$ is a compact open subgroup of $G$ and $H_0 = H \cap G_0$. Then 
\[ D(G)_H = D(G_0)_{H_0} \otimes D(H_0 \backslash H) \]
as left $D(G_0)_{H_0}$-modules.
\end{lemma}

\begin{corollary} \label{change-of-compact}
Suppose $G_0 \subseteq G_1$ are two compact open subgroups of $G$, and let $H_i = H \cap G_i$. Then $D(G_1)_{H_1}$ is finite free as a left module over $D(G_0)_{H_0}$. 
\end{corollary}

\begin{proof}
\Cref{compact-direct-sum} tells us that $D(G_1)_{H_1} = D(G_0)_{H_0} \otimes D(H_0 \backslash H_1)$. Since $H_1$ is compact and $H_0$ is open, we see that $H_0 \backslash H_1$ is finite, so $D(H_0 \backslash H_1)$ is finite dimensional. 
\end{proof}

\begin{definition}
A $D(G)_H$-module $M$ is \emph{coadmissible} if there exists a compact open subgroup $G_0$ of $G$ such that $M$ is coadmissible as a module over the Fr\'echet-Stein algebra $D(G_0)_{H_0}$, where $H_0 = H \cap G_0$. 
By \cref{change-of-compact} and \cite[lemma 3.8]{ST03}, it follows that the same is then true for any other compact open subgroup. 
\end{definition}

\subsection{Pre-coadmissibile modules over \texorpdfstring{$D(\frg,H)$}{D(g,H)}}

\begin{lemma}
If $G_0$ is a compact open subgroup of $G$ and $H_0 = H \cap G_0$, then 
\[ D(G_0)_{H_0} \otimes_{D(\frg, H_0)} D(\frg, H) = D(G)_H \]
as $D(G_0)_{H_0}$-$D(\frg,H)$-bimodules.
\end{lemma}

\begin{proof}
Observe that 
\[ \begin{aligned} D(\frg, H) &= U(\frg) \otimes_{U(\frh)} D(H) \\
&= U(\frg) \otimes_{U(\frh)} \left( D(H_0) \otimes D(H_0 \backslash H) \right) \\
&= \left( U(\frg) \otimes_{U(\frh)} D(H_0) \right) \otimes D(H_0 \backslash H) \\
&= D(\frg, H) \otimes D(H_0 \backslash H). \end{aligned} \]
so
\[ D(G_0)_{H_0} \otimes_{D(\frg, H_0)} D(\frg, H) = D(G_0)_{H_0} \otimes D(H_0 \backslash H) = D(G)_H, \]
using \cref{compact-direct-sum} for the last step. 
\end{proof}

\begin{corollary} \label{base-change-compact}
Suppose $M$ is a $D(\frg,H)$-module. For any compact open subgroup $G_0$ of $G$, \[ D(G)_H \otimes_{D(\frg,H)} M = D(G_0)_{H_0} \otimes_{D(\frg,H_0)} M \] as $D(G_0)_{H_0}$-modules. \qed
\end{corollary}

\begin{definition}
A $D(\frg,H)$-module $M$ is \emph{pre-coadmissible} if $D(G)_H \otimes_{D(\frg,H)} M$ is coadmissible. 
\end{definition}

Suppose $M$ is a $D(\frg,H)$-module. If $G_0$ is a compact open subgroup of $G$, $H_0 = H \cap G_0$, and $M$ is finitely presented over $D(\frg,H_0)$, then $M$ is pre-coadmissible. Also, if $M$ is finitely generated over $U(\frg)$, it is pre-coadmissible (by the same proof as in \cite[lemma 4.3]{ScSt}, mutatis mutandis). 

\begin{lemma}[{\cite[lemma 4.6]{ScSt}}] \label{two-different-base-changes}
For a $D(\frg,H)$-module $M$, the natural map
\[ \begin{tikzcd} U_r(\frg) \otimes_{U(\frg)} M \ar{r} & D_r(G)_H \otimes_{D(\frg,H)} M \end{tikzcd} \]
is an isomorphism of left $U_r(\frg)$-modules. 
\end{lemma}

\begin{proposition} \label{exactness-DgP-DGP}
If 
\begin{equation} \label{exact-DgP} \begin{tikzcd} 0 \ar{r} & M' \ar{r} & M \ar{r} & M'' \ar{r} & 0 \end{tikzcd} \end{equation}
is an exact sequence of pre-coadmissible $D(\frg,H)$-modules, then \[ D(G)_H \otimes_{D(\frg,H)} \eqref{exact-DgP} \] is an exact sequence of coadmissible $D(G)_H$-modules. 
\end{proposition}

\begin{proof}
Using \cref{base-change-compact}, we can assume that $G$ is compact. Then \[ D(G)_{H} \otimes_{D(g,H)} \eqref{exact-DgP} \]
is a sequence of coadmissible modules over the Fr\'echet-Stein algebra $D(G)_{H}$, so it is exact if and only if
\[ D_r(G)_{H} \otimes_{D(G)_{H}} \left(D(G)_{H} \otimes_{D(\frg,H)} \eqref{exact-DgP} \right) = D_r(G)_{H} \otimes_{D(\frg,H)} \eqref{exact-DgP} \]
is exact for a cofinal set of values of $r$. But 
\[ D_r(G)_{H} \otimes_{D(\frg,H)} \eqref{exact-DgP} = U_r(\frg) \otimes_{U(\frg)} \eqref{exact-DgP} \]
by \cref{two-different-base-changes}. Furthermore, we know that $U_r(\frg)$ is flat over $U(\frg)$, since the Arens-Michael envelope $\hat{U}(\frg) = \lim_r U_r(\frg)$ is flat over $U(\frg)$ \cite[theorem 3.13]{ScSt} and $U_r(\frg)$ is flat over $\hat{U}(\frg)$ \cite[remark 3.2]{ST03}. Thus $U_r(\frg) \otimes_{U(\frg)} \eqref{exact-DgP}$ is in fact exact for all $r$. 
\end{proof}

\begin{corollary} \label{exactness-fp}
Suppose $G$ is compact. If \eqref{exact-DgP} is an exact sequence of pre-coadmissible $D(\frg,H)$-modules. Then 
\[ D(G) \arcotimes_{D(G)_H} \left( D(G)_H \otimes_{D(\frg,H)} \eqref{exact-DgP} \right) \]
is an exact sequence of coadmissible $D(G)$-modules. In particular, if \eqref{exact-DgP} is an exact sequence of finitely presented $D(\frg,H)$-modules, then \[ D(G) \otimes_{D(\frg,H)} \eqref{exact-DgP} \] is an exact sequence of coadmissible $D(G)$-modules. 
\end{corollary}

\begin{proof}
The first part follows from \cref{distributions-flat-over-supported-distributions,exactness-DgP-DGP}. For the second part, suppose $M$ is a finitely presented $D(\frg,H)$-module. Then \[ D(G)_H \otimes_{D(\frg,H)} M \] is also finitely presented, and $D(G) \arcotimes_{D(G)_H}\, -$ agrees with $D(G) \otimes_{D(G)_H} -$ for finitely presented $D(G)_H$-modules. 
\end{proof}